\documentclass[12pt, reqno]{amsart}
\usepackage{amsmath, amsthm, amscd, amsfonts, amssymb}

\usepackage{amssymb}
\usepackage{amsfonts}
\usepackage{amsmath}
\usepackage[all,cmtip]{xy}

\usepackage{graphicx}

\usepackage{enumerate}

\usepackage{tikz}
\usetikzlibrary{arrows}
\usetikzlibrary{automata,positioning}
\usetikzlibrary{calc,shapes.arrows}

\newtheorem{theorem}{Theorem}[section]
\newtheorem{lemma}[theorem]{Lemma}
\newtheorem{proposition}[theorem]{Proposition}
\newtheorem{corollary}[theorem]{Corollary}
\theoremstyle{definition}
\newtheorem{definition}[theorem]{Definition}
\newtheorem{example}[theorem]{Example}

\theoremstyle{remark}
\newtheorem{remark}[theorem]{Remark}
\numberwithin{equation}{section}

\newcommand{\C}{$C^\ast$}
\newcommand{\G}{\mathcal{G}}

\usepackage{datetime}

\begin{document}

\title[Relative ultragraph algebras and infinite interval maps]{Relative ultragraph algebras and infinite interval maps}
   
 \author{Ben-Hur Eidt, Daniel Gon\c{c}alves, Danilo Royer}

\address{Instituto de Ciências Matemáticas e de Computação, Universidade de São Paulo, São Carlos, 13566-590\\
 Brazil}
   \email{benhur96dt@gmail.com}  
 
\address{Departamento de Matem\'atica,
Universidade Federal de Santa Catarina,
Florian\'opolis, 88040-900\\
 Brazil}
   \email{daemig@gmail.com}  

\address{Departamento de Matem\'atica,
Universidade Federal de Santa Catarina,
Florian\'opolis, 88040-900\\
 Brazil}
   \email{daniloroyer@gmail.com}

\begin{abstract}

We introduce the notion of relative ultragraph $C^*$-algebras and extend classical injectivity criteria for representations—especially those arising from branching systems—to this relative setting.

We next consider interval maps with a countable Markov partition (henceforth, “Markov infinite maps”). Each such map determines a countably infinite transition matrix and thereby an ultragraph $\mathcal{G}$; under suitable hypotheses this yields representations of  relative ultragraph $C^*$-algebras $C^*(\mathcal{G},X)$. For one-dimensional systems whose underlying Markov infinite map admits an escape set, we show that the representation of $C^*(\mathcal{G},X)$ built from point orbits coincides with the representation arising from an appropriate relative branching system.
Finally, we apply our injectivity criteria to these representations, obtaining concrete conditions under which they are faithful.

\end{abstract}

\maketitle



\section{Introduction}

The study of $C^*$-algebras associated to directed graphs and their generalizations is a central theme at the interface of operator algebras and dynamics. Among these generalizations, \emph{ultragraph} $C^*$-algebras (see, e.g., \cite{Tom3,Katsura, GilleseDaniel}) extend the classical Cuntz--Krieger and Exel--Laca frameworks, while \emph{relative} graph $C^*$-algebras \cite{Tom} provide a flexible setting interpolating between Toeplitz and graph algebras (see also \cite{CarlsenandLarsen, Sims}). 

A complementary perspective is offered by \emph{branching systems}. Originating in work on Cuntz–Krieger algebras, dynamical systems, and wavelets (e.g., \cite{AK, BJ,DuPalle,marcolli}), branching and semibranching systems have since underpinned a broad program connecting symbolic dynamics to concrete representations of graph-like $C^*$-algebras. In the graph setting, a systematic representation theory via branching systems has been developed, including unitary equivalence \cite{Daniel} and further structural developments \cite{DDinicio, GLR2016CMB}. The approach has been extended to related algebraic frameworks (e.g., relative Cohn path algebras \cite{danielcristobal, DDseparated}) and lifted to higher-rank graphs \cite{GLR2018GMJ}, as well as to separable/monic representations and wavelet-type constructions \cite{FGKP2016,FGJKP2020}. These developments make branching systems a natural language for building and analyzing concrete representations.

In this paper, we introduce and develop \emph{relative ultragraph} $C^*$-algebras, and we show that, under a mild additional hypothesis, a large class of them are (canonically) isomorphic to ultragraph $C^*$-algebras (Theorem~\ref{teo iso algebras ultragrafos e relativas}). This parallels the graph case \cite{Tom} and has algebraic analogues for relative Cohn and Leavitt path algebras \cite{Livroversaoalgebrica}. 

Motivated by classical (finite) Markov maps on intervals (cf.\ \cite{CMP,RMP5,RMP10} and the broader Markov-map literature \cite{Rufus,OutroMarkov}), we formulate a notion of \emph{infinite partition Markov interval maps with escape sets}. Each such map $g$ determines a $0$--$1$ incidence matrix and hence an ultragraph $\mathcal{G}$; we then construct representations of suitable relative ultragraph algebras $C^*(\mathcal{G},X)$ inside $B(\ell^2(R_g(x)))$ directly from the dynamics of $g$ (Theorem ~\ref{basic} and Proposition \ref{repMarkov}). A key point is that the dichotomy “return to the domain vs.\ escape’’ for the iterates of $g$ naturally produces the relative relations. 

A second main contribution is to bridge these \emph{Markov-induced} representations with branching systems: under natural hypotheses, every representation built from $g$ coincides with one induced by a $(\mathcal{G},X)$-relative branching system (Theorem~\ref{markovandbsareequal}). This identification lets us import injectivity criteria and uniqueness theorems from the branching-system toolkit to the Markov setting with minimal overhead (see Theorem~\ref{injectivityformarkovreps} and its applications).

Our results thus establish a new bridge between interval dynamics and ultragraph operator algebras, expanding the reach of the relative framework and offering new tools for the analysis of infinite-state dynamical systems.

The paper is organized as follows.
Section~2 defines relative ultragraph $C^*$-algebras and proves the isomorphism with suitable ultragraph algebras. Section~3 adapts uniqueness theorems to the relative context. Section~4 develops relative branching systems and their induced representations, establishing the connection with ultragraph representations. Section~5 introduces infinite Markov interval maps with escape sets and builds the associated representations, proving their equivalence with branching-system representations. Section~6 presents examples illustrating the theory and injectivity criteria in concrete settings.

\section{Relative ultragraph $C^\ast$-algebras}
\label{secgraphalgs}

\subsection{Background on ultragraphs and its C*-algebras}

In this subsection we recall the main definitions and relevant results regarding ultragraphs, as introduced in \cite{Tom2}.

\begin{definition}\label{def of ultragraph}
An \emph{ultragraph} is a quadruple $\mathcal{G}=(\G^0, \mathcal{G}^1, r,s)$ consisting of two countable sets $\G^0, \mathcal{G}^1$, a map $s:\mathcal{G}^1 \to \G^0$, and a map $r:\mathcal{G}^1 \to P(\G^0)$, where $P(\G^0)$ stands for the power set of $\G^0$.
\end{definition}



Before we define the C*-algebra associated to an ultragraph we need the following notion.

\begin{definition}\label{def of mathcal{G}^0}
Let $\mathcal{G}$ be an ultragraph. Define $\mathcal{E}$ to be the smallest subset of $P(G^0)$ that contains $\{v\}$ for all $v\in G^0$, contains $r(e)$ for all $e\in \mathcal{G}^1$, and is closed under finite unions and finite intersections. The collection $\mathcal{E}$ is called the boolean algebra associated to the ultragraph $\G$.
\end{definition}


\begin{definition}\label{def of C^*(mathcal{G})}
Let $\mathcal{G}$ be an ultragraph. The \emph{ultragraph algebra} $C^*(\mathcal{G})$ is the universal $C^*$-algebra generated by a family of partial isometries with orthogonal ranges $\{s_e:e\in \mathcal{G}^1\}$ and a family of projections $\{p_A:A\in \mathcal{E}\}$ satisfying
\begin{enumerate}
\item\label{p_Ap_B=p_{A cap B}}  $p_\emptyset=0,  p_Ap_B=p_{A\cap B},  p_{A\cup B}=p_A+p_B-p_{A\cap B}$, for all $A,B\in \mathcal{E}$;
\item\label{s_e^*s_e=p_{r(e)}}$s_e^*s_e=p_{r(e)}$, for all $e\in \mathcal{G}^1$;
\item $s_es_e^*\leq p_{s(e)}$ for all $e\in \mathcal{G}^1$; and
\item\label{CK-condition} $p_v=\sum\limits_{s(e)=v}s_es_e^*$ whenever $0<\vert s^{-1}(v)\vert< \infty$.
\end{enumerate}
\end{definition}


\subsection{Relative ultragraph C*-algebras}

Let $\G$ be an ultragraph and denote by $Reg(\G)$ the set of regular vertices, that is, the set of vertices $v$ such that $0<\vert s^{-1}(v)\vert< \infty$. We have the following definition.

\begin{definition}\label{relative}
We say that  a pair $(\mathcal{G},X)$ is a relative ultragraph if $\G$ is an ultragraph and $X$ is any subset of $Reg(\G)$. The relative ultragraph \C-algebra $C^*(\G,X)$ is the universal 
$C^*$-algebra generated by a family of partial isometries with orthogonal ranges $\{s_e:e\in \mathcal{G}^1\}$ and a family of projections $\{p_A:A\in \mathcal{E}\}$ satisfying the conditions (1) to (3) of Definition~\ref{def of C^*(mathcal{G})} and satisfying, for every vertex $v\in X$, condition (4) of Definition~\ref{def of C^*(mathcal{G})}.
\end{definition}

Notice that for a given relative ultragraph $(\mathcal{G}, X)$, the ultragraph 
$C^*$-algebra $C^*(\mathcal{G})$ can be realized as a quotient of the relative 
ultragraph $C^*$-algebra $C^*(\mathcal{G}, X)$. We now describe a construction 
that associates to a relative ultragraph an ordinary ultragraph in such a way 
that, under suitable conditions, the corresponding $C^*$-algebras are isomorphic. 
This construction parallels the one for graphs presented in 
\cite[Theorem~3.7]{Tom}.




\begin{definition}\label{ultragrap GX}
Let $(\mathcal{G},X)$ be a relative ultragraph, and set 
\[
Y := \mathrm{Reg}(\mathcal{G}) \setminus X, 
\quad 
Y' := \{\, y' \mid y \in Y \,\},
\]
where $Y'$ denotes a disjoint copy of $Y$.  
We define the ultragraph $\mathcal{G}_X = (\mathcal{G}^0_X, \mathcal{G}^1_X, r_X, s_X)$ as follows:
\[
\mathcal{G}_X^0 := \mathcal{G}^0 \cup \{\, v' \mid v \in Y \,\},
\ \ \ \mathcal{G}_X^1 := \mathcal{G}^1 \cup \{\, e' \mid e \in \mathcal{G}^1 \text{ and } r(e) \cap Y \neq \emptyset \,\}.
\]
The maps $s_X$ and $r_X$, extending $s$ and $r$ to $\mathcal{G}_X^1$, are defined in the elements of $\mathcal{G}_X^1$ that  are not contained in $\mathcal{G}^1$ by
\[
s_X(e') := s(e), 
\quad 
r_X(e') := (r(e) \cap Y)' = \{\, w' \mid w \in r(e) \cap Y \,\}.
\]
\end{definition}

\begin{remark}
If $A,B \subseteq Y$, then 
\[
(A \cap B)' = A' \cap B',
\quad 
(A \cup B)' = A' \cup B'.
\]
\end{remark}

We shall use the notation
\[
C := \{\, e' \mid e \in \mathcal{G}^1 \text{ and } r(e) \cap Y \neq \emptyset \,\},
\]
so that $\mathcal{G}_X^1 = \mathcal{G}^1 \cup C$.  
Moreover, let $\mathcal{E}$ denote the Boolean algebra of the ultragraph $\mathcal{G}$, and let $\mathcal{E}_X$ denote the Boolean algebra of the ultragraph $\mathcal{G}_X$.  
The following lemma describes the relation between these two Boolean algebras.

\begin{lemma}
Let $(\mathcal{G}, X)$ be a relative ultragraph, and set $Y := \mathrm{Reg}(\mathcal{G}) \setminus X$.
\begin{enumerate}
    \item For every $A \in \mathcal{E}$ we have $(A \cap Y)' \in \mathcal{E}_X$.
    \item For every $B \in \mathcal{E}_X$ we have $B \cap \mathcal{G}^0 \in \mathcal{E}$.
\end{enumerate}
\end{lemma}

\begin{proof}
\begin{enumerate}
    \item Let $A \in \mathcal{E}$. By \cite[Lemma~2.12]{Tom2}, we may write
    \[
    A = \bigcup_{i=1}^N \left( \bigcap_{e \in X_i} r(e) \right) \cup F,
    \]
    where $F \subseteq \mathcal{G}^0$ is finite and each $X_i \subseteq \mathcal{G}^1$ is finite.  
    Using Remark~\ref{ultragrap GX}, we obtain
    \begin{align*}
        (A \cap Y)' 
        &= \bigcup_{i=1}^N \left( \bigcap_{e \in X_i} (r(e) \cap Y)' \right) \cup (F \cap Y)' \\
        &= \bigcup_{i=1}^N \left( \bigcap_{e \in X_i \cap Q} r(e') \right) \cup (F \cap Y)',
    \end{align*}
    where 
    \[
    Q := \{\, f \in \mathcal{G}^1 : r(f) \cap Y \neq \emptyset \,\}.
    \]
    (Note that if $f \notin Q$, then $r(f) \cap Y = \emptyset$.)  
    Since each $X_i \cap Q$ is finite and $(F \cap Y)'$ is finite, \cite[Lemma~2.12]{Tom2} implies that $(A \cap Y)' \in \mathcal{E}_X$.

    \item Let $B \in \mathcal{E}_X$. By \cite[Lemma~2.12]{Tom2}, we may write
    \[
    B = \bigcup_{i=1}^N \left( \bigcap_{f \in X_i} r_X(f) \right) \cup F,
    \]
    where $F \subseteq \mathcal{G}^0_X$ is finite and each $X_i \subseteq \mathcal{G}_X^1$ is finite.  
    Intersecting with $\mathcal{G}^0$ gives
    \[
    B \cap \mathcal{G}^0 = \bigcup_{i=1}^N \left( \bigcap_{f \in X_i} (r_X(f) \cap \mathcal{G}^0) \right) \cup (F \cap \mathcal{G}^0).
    \]
    Recall that $\mathcal{G}_X^1 = \mathcal{G}^1 \cup C$.  
    If $f \in C$, then $r_X(f) \cap \mathcal{G}^0 = \emptyset$, while if $f \in \mathcal{G}^1$ we have $r_X(f) = r(f) = r(f) \cap \mathcal{G}^0$.  
    Hence
    \[
    B \cap \mathcal{G}^0 = \bigcup_{i=1}^N \left( \bigcap_{f \in X_i \cap \mathcal{G}^1} r(f) \right) \cup (F \cap \mathcal{G}^0).
    \]
    By \cite[Lemma~2.12]{Tom2}, it follows that $B \cap \mathcal{G}^0 \in \mathcal{E}$.
\end{enumerate}
\end{proof}

In the following result we construct a natural $*$-homomorphism from the relative ultragraph algebra 
$C^*(\mathcal{G},X)$ into the ultragraph algebra $C^*(\mathcal{G}_X)$. We denote by $s_e$ and $p_A$ the canonical generators of 
$C^*(\mathcal{G},X)$, where $e \in \mathcal{G}^1$ and $A \in \mathcal{E}$, 
and by $S_f$ and $P_B$ the canonical generators of $C^*(\mathcal{G}_X)$, 
where $f \in \mathcal{G}_X^1$ and $B \in \mathcal{E}_X$.

\begin{proposition}
    \label{esfriou, vai dar tainha}
Let $(\mathcal{G}, X)$ be a relative ultragraph. Then there exists a 
$*$-homomorphism 
\[
\phi: C^*(\mathcal{G},X) \longrightarrow C^*(\mathcal{G}_X)
\]
such that, for each $B \in \mathcal{E}$,
\[
\phi(p_B) = P_B + P_{(B \cap Y)'}
\]
and, for each $e \in \mathcal{G}^1$,
\[
\phi(s_e) =
\begin{cases}
    S_e, & \text{if } r(e) \cap Y = \emptyset, \\[6pt]
    S_e + S_{e'}, & \text{if } r(e) \cap Y \neq \emptyset.
\end{cases}
\]
\end{proposition}

\begin{proof}
We verify that 
\[
\{\phi(p_B): B \in \mathcal{E}\}\ \cup\ \{\phi(s_e): e \in \mathcal{G}^1\}
\]
satisfies the defining relations of the universal $C^*$-algebra $C^*(\mathcal{G},X)$. 
By the universal property, $\phi$ then extends to a $*$-homomorphism 
$C^*(\mathcal{G},X)\to C^*(\mathcal{G}_X)$.

\begin{enumerate}
\item \textbf{Boolean relations.}
Clearly $\phi(p_{\emptyset})=P_{\emptyset}=0$. For $A,B\in\mathcal{E}$,
\begin{align*}
\phi(p_A)\phi(p_B)
&=(P_A+P_{(A\cap Y)'})(P_B+P_{(B\cap Y)'})\\
&=P_AP_B+P_AP_{(B\cap Y)'}+P_{(A\cap Y)'}P_B+P_{(A\cap Y)'}P_{(B\cap Y)'}\\
&=P_{A\cap B}+P_{(A\cap B\cap Y)'}\\
&=\phi(p_{A\cap B}),
\end{align*}
so $\phi(p_A)\phi(p_B)=\phi(p_{A\cap B})$. Similarly,
\[
\phi(p_A)+\phi(p_B)-\phi(p_{A\cap B})=\phi(p_{A\cup B}).
\]

\item \textbf{Range relations.}
We show $\phi(p_{r(e)})=\phi(s_e)^*\phi(s_e)$ for all $e\in\mathcal{G}^1$.

If $r(e)\cap Y=\emptyset$, then $\phi(s_e)=S_e$ and
\[
\phi(s_e)^*\phi(s_e)=S_e^*S_e=P_{r(e)}=\phi(p_{r(e)}).
\]
If $r(e)\cap Y\neq\emptyset$, then
\begin{align*}
\phi(p_{r(e)})&=P_{r(e)}+P_{(r(e)\cap Y)'}\\
&=P_{r_X(e)}+P_{r_X(e')}\\
&=S_e^*S_e+S_{e'}^*S_{e'},
\end{align*}
while
\begin{align*}
\phi(s_e)^*\phi(s_e)
&=(S_e+S_{e'})^*(S_e+S_{e'})\\
&=S_e^*S_e+S_e^*S_{e'}+S_{e'}^*S_e+S_{e'}^*S_{e'}\\
&=S_e^*S_e+S_{e'}^*S_{e'},
\end{align*}
since $S_f^*S_g=0$ whenever $f\ne g\in\mathcal{G}_X^1$. Hence $\phi(p_{r(e)})=\phi(s_e)^*\phi(s_e)$.

\item \textbf{Source projection relation.}
For every $e\in\mathcal{G}^1$,
\[
\phi(p_{s(e)})\,\phi(s_e)\phi(s_e)^*=\phi(s_e)\phi(s_e)^*.
\]
Indeed, write $\phi(p_{s(e)})=P_{s(e)}+\mathbf{1}_{\{s(e)\in Y\}}\,P_{s(e)'}$.  
If $r(e)\cap Y=\emptyset$, then $\phi(s_e)\phi(s_e)^*=S_eS_e^*\le P_{s(e)}$ and $P_{s(e)'}S_eS_e^*=0$, so the equality holds.  
If $r(e)\cap Y\neq\emptyset$, then $\phi(s_e)\phi(s_e)^*=S_eS_e^*+S_{e'}S_{e'}^*$ and each term is $\le P_{s(e)}$, while again $P_{s(e)'}$ annihilates them. Thus the relation holds in all cases.

\item \textbf{Cuntz--Krieger relation at $v\in X$.}
Fix $v\in X$. Since $\{v\}\cap Y=\emptyset$, we have $\phi(p_v)=P_v$. Moreover,
\begin{align*}
P_v
&=\sum_{\{f\in\mathcal{G}_X^1:\ s_X(f)=v\}} S_fS_f^* = \sum_{\{e\in\mathcal{G}^1:\ s(e)=v\}} S_eS_e^*
\;+\;\sum_{\{e'\in\mathcal{G}_X^1\setminus\mathcal{G}^1:\ s_X(e')=v\}} S_{e'}S_{e'}^*\\
&= \sum\limits_{ \{ e \in \G^1 \ | \ r(e) \cap Y = \emptyset \textrm{ and } s(e) = v \} } S_e S_e^* + \sum\limits_{ \{ e \in \G^1 \ | \ r(e) \cap Y \neq \emptyset \textrm{ and } s(e) = v \} } S_e S_e^* + S_{e'}S_{e'}^*.
\end{align*}
For $e$ with $s(e)=v$ we have, by direct computation,
\[
\phi(s_e)\phi(s_e)^*=
\begin{cases}
S_eS_e^*, & r(e)\cap Y=\emptyset,\\[4pt]
S_eS_e^*+S_{e'}S_{e'}^*, & r(e)\cap Y\neq\emptyset.
\end{cases}
\]
Summing over all $e$ with source $v$ gives
\[
P_v=\sum_{\{e\in\mathcal{G}^1:\ s(e)=v\}} \phi(s_e)\phi(s_e)^*,
\]
which is the required Cuntz--Krieger relation at $v$.
\end{enumerate}

Therefore all defining relations of $C^*(\mathcal{G},X)$ are satisfied, and by universality $\phi$ extends to a $*$-homomorphism
\[
\phi: C^*(\mathcal{G},X)\longrightarrow C^*(\mathcal{G}_X).
\]
\end{proof}

The next lemma will be used later to show that, under an additional hypothesis, the
homomorphism $\phi$ constructed above is an isomorphism.

\begin{lemma}\label{decomposition1}
Let $(\mathcal{G},X)$ be a relative ultragraph and set $Y:=\mathrm{Reg}(\mathcal{G})\setminus X$.
If $Z \in \mathcal{E}_X$, then there exist $A,B \in \mathcal{E}$ such that
\[
Z \;=\; A \;\cup\; (B \cap Y)'.
\]
Moreover, $A$ and $B \cap Y$ are uniquely determined.
\end{lemma}

\begin{proof}
By \cite[Lemma~2.12]{Tom2}, we can write
\[
Z \;=\; \bigcup_{i=1}^N \left(\;\bigcap_{f \in X_i} r_X(f)\;\right) \;\cup\; F,
\]
where each $X_i \subseteq \mathcal{G}_X^1$ is finite and $F\subseteq \mathcal{G}_X^0$ is finite.
We may assume (discarding empty intersections) that for each $i$ either
$X_i \subseteq \mathcal{G}^1$ or $X_i \subseteq C$, where
\[
C \;:=\; \{\, e' \mid e \in \mathcal{G}^1,\ r(e)\cap Y \neq \emptyset \,\}.
\]
Indeed, if $X_i$ contains both some $f\in\mathcal{G}^1$ and some $e'\in C$, then
$r_X(f)=r(f)\subseteq \mathcal{G}^0$ while $r_X(e')=(r(e)\cap Y)'\subseteq Y'$ are disjoint, so
$\bigcap_{g\in X_i} r_X(g)=\emptyset$.

Reindex so that for some $K\in\{0,\dots,N\}$ we have
$X_1,\dots,X_K \subseteq \mathcal{G}^1$ and $X_{K+1},\dots,X_N \subseteq C$.
Set
\[
A \;:=\; \bigcup_{i=1}^K \left(\;\bigcap_{f \in X_i} r(f)\;\right)\ \cup\ (F \cap \mathcal{G}^0).
\]
Then $A\in \mathcal{E}$, because it is obtained from finite vertex sets in $\mathcal{G}^0$ and ranges
$r(e)$ with $e\in\mathcal{G}^1$ using finite unions and intersections.

For the remaining part, note that for $i\ge K+1$ and $X_i\subseteq C$ we can write
$X_i=\{e'_1,\dots,e'_m\}$ with $e_j\in\mathcal{G}^1$, so using the properties of $(\cdot)'$,
\[
\bigcap_{e'\in X_i} r_X(e') \;=\; 
\bigcap_{e'\in X_i} (r(e)\cap Y)' \;=\;
\left(\;\bigcap_{e'\in X_i} (r(e)\cap Y)\;\right)'.
\]
Hence
\begin{align*}
Z \setminus A
&= \bigcup_{i=K+1}^N \left(\;\bigcap_{e'\in X_i} r_X(e')\;\right) \ \cup\ (F\cap Y') \\
&= \left(\ \bigcup_{i=K+1}^N \bigcap_{e'\in X_i} (r(e)\cap Y)\ \right)'\ \cup\ (F\cap Y') \\
&= \left(\ \Big(\;\bigcup_{i=K+1}^N \bigcap_{e'\in X_i} r(e)\;\Big)\ \cap Y\ \right)'\ \cup\ (F\cap Y').
\end{align*}
Define
\[
B_1 \;:=\; \bigcup_{i=K+1}^N \bigcap_{e'\in X_i} r(e) \ \in\ \mathcal{E}.
\]
Since $F\cap Y'$ is finite, write $F\cap Y'=\{v_1',\dots,v_n'\}$ and set
\[
B \;:=\; B_1 \cup \{v_1,\dots,v_n\} \ \in\ \mathcal{E}.
\]
Then $(B\cap Y)'=(B_1\cap Y)'\cup \{v_1',\dots,v_n'\}$, hence
$Z\setminus A=(B\cap Y)'$. As $A\subseteq \mathcal{G}^0$ and $(B\cap Y)'\subseteq Y'$ are
disjoint, we conclude $Z=A\cup(B\cap Y)'$.

For uniqueness, suppose $Z=A_1\cup (B_1\cap Y)'=A_2\cup (B_2\cap Y)'$ with $A_i,B_i\in\mathcal{E}$.
Intersecting with $\mathcal{G}^0$ gives $A_1=Z\cap \mathcal{G}^0=A_2$. 
Likewise, applying the inverse of $(\cdot)'$ on $Y'$ gives
$B_1\cap Y=\{\,w\in Y:\ w'\in Z\,\}=B_2\cap Y$. Thus $A$ and $B\cap Y$ are uniquely determined.
\end{proof}

\begin{corollary}\label{decomposition2}
Let $(\mathcal{G},X)$ be a relative ultragraph and set $Y=\mathrm{Reg}(\mathcal{G})\setminus X$.
If $r(e)\cap Y$ is finite for every $e\in\mathcal{G}^1$, then for each $Z\in\mathcal{E}_X$ there
is a unique decomposition
\[
Z \;=\; A \;\cup\; (B\cap Y)',
\]
with $A,B\in\mathcal{E}$ and $(B\cap Y)'$ finite.
\end{corollary}

\begin{proof}
With the notation in the proof of Lemma~\ref{decomposition1}, the hypothesis implies
that each set $\bigcap_{e'\in X_i}(r(e)\cap Y)$ (with $X_i\subseteq C$) is finite; hence their
finite union is finite. Therefore $(B_1\cap Y)'$ is finite, and after adjoining the
finitely many points coming from $F\cap Y'$ we still have $(B\cap Y)'$ finite.
Uniqueness is as in Lemma~\ref{decomposition1}.
\end{proof}

Our next goal is to prove that, if $r(e)\cap Y$ is finite for every $e\in\mathcal{G}^1$, then
$\phi$ is a $*$-isomorphism $C^*(\mathcal{G},X)\cong C^*(\mathcal{G}_X)$. 
We split the argument into two propositions: the first establishes \emph{surjectivity}, and the
second constructs an explicit left-inverse $\psi:C^*(\mathcal{G}_X)\to C^*(\mathcal{G},X)$ with
$\psi\circ\phi=\mathrm{id}$, which yields \emph{injectivity} of $\phi$.

\begin{proposition}\label{surjective homomorphism}
Let $(\mathcal{G},X)$ be a relative ultragraph and set $Y:=\mathrm{Reg}(\mathcal{G})\setminus X$.
If $r(e)\cap Y$ is finite for all $e\in\mathcal{G}^1$, then the $*$-homomorphism
$\phi:C^*(\mathcal{G},X)\to C^*(\mathcal{G}_X)$ from Proposition~\ref{esfriou, vai dar tainha} is
surjective.
\end{proposition}

\begin{proof}
We show that the canonical generators of $C^*(\mathcal{G}_X)$ lie in $\operatorname{Im}(\phi)$.
Since $C^*(\mathcal{G}_X)$ is generated by $\{P_B : B\in \mathcal{E}_X\}\cup\{S_f : f\in \mathcal{G}_X^1\}$,
this will prove surjectivity.
     
First, we show that all projections \(P_w\), for \(w \in E^0_X\), are contained in \(\mathrm{Im}(\phi)\). If \(w = v \in \G^0\) and \(v \notin Y\), then \(\phi(p_v) = P_v\); if \(w = v \in \G^0\) and \(v \in Y\), then

 \begin{align*}
    \sum\limits_{e \in \G^1 \, | \, s(e) = v} \phi(s_e)\phi(s_e)^* 
    & = \sum\limits_{\substack{e \in \G^1 \, | \, s(e) = v \\ r(e) \cap Y = \emptyset}} \phi(s_e)\phi(s_e)^* 
    + \sum\limits_{\substack{e \in \G^1\, | \, s(e) = v \\ r(e) \cap Y \neq \emptyset}} \phi(s_e)\phi(s_e)^* \\
    & = \sum\limits_{\substack{e \in \G^1\, | \, s(e) = v \\ r(e) \cap Y = \emptyset}} s_es_e^* 
    + \sum\limits_{\substack{e \in \G^1 \, | \, s(e) = v \\ r(e) \cap Y \neq \emptyset}} (s_e + s_{e'})(s_e + s_{e'})^* \\
    & = \sum\limits_{\substack{e \in \G^1 \, | \, s(e) = v \\ r(e) \cap Y = \emptyset}} s_es_e^* 
    + \sum\limits_{\substack{e \in \G^1 \, | \, s(e) = v \\ r(e) \cap Y \neq \emptyset}} \big(s_es_e^* + s_{e'} s_{e'}^*\big) \\
    & = \sum\limits_{f \in \G^1_X \, | \, s(f) = v} s_fs_f^* = P_v.
\end{align*}

Thus, $P_v \in Im(\phi)$ and, by definition,
    $$\phi \left( p_v - \sum\limits_{e \in E^1 \ | \ s(e)  = v} s_es_e^*  \right) = P_v + P_{v'} - P_v = P_{v'}, $$
showing that all projections $P_w$ are in the image of $\phi$. 

Next, we prove that all projections of the form \(P_{r(f)}\), for \(f \in \G_X^1\), belong to \(\text{Im}(\phi)\). If \(f \in C\), then \(f = e'\) for some \(e \in E^1\), and \(r(f) = r(e') = (r(e) \cap Y)'\), which is finite by hypothesis. Therefore, \(P_{r(f)} = \sum_{v \in r(f)} P_v \in \text{Im}(\phi)\). If \(f \in E^1\), we consider two cases: if \(r(f) \cap Y = \emptyset\), then \(\phi(P_{r(f)}) = P_{r(f)}\); if \(r(f) \cap Y \neq \emptyset\), then \(\phi(P_{r(f)}) = P_{r(f)} + P_{(r(f) \cap Y)'} = P_{r(f)} + P_{r(f')}\), showing that \(P_{r(f)} = \phi(P_{r(f)}) - P_{r(f')} \in \text{Im}(\phi)\).

To conclude the proof, it suffices to show that all elements of the form \(S_f\), for \(f \in \G_X^1\), are in the image of \(\phi\). This is clear if \(f \in E^1\) and \(r(f) \cap Y = \emptyset\) because \(\phi(s_f) = S_f\). If \(f \in E^1\) and \(r(f) \cap Y \neq \emptyset\), we observe that $
\phi(s_f) P_{r(f)} = (S_f + S_{f'}) P_{r(f)} = S_f$ and $
\phi(s_f) P_{r(f')} = (S_f + S_{f'}) P_{r(f')} = S_{f'}
$.
Since \(P_{r(f)}\) and \(P_{r(f')}\) are in \(\text{Im}(\phi)\), we conclude that \(S_f, S_{f'} \in \text{Im}(\phi)\), completing the proof.

\end{proof}

To clarify the proof of injectivity of $\phi$, we begin with a computational lemma.

\begin{lemma}\label{contas1}
Let $(\mathcal{G},X)$ be a relative ultragraph and set $Y:=\mathrm{Reg}(\mathcal{G})\setminus X$.
Suppose $Z_1=A_1\cup (B_1\cap Y)'$ and $Z_2=A_2\cup (B_2\cap Y)'$ with $A_1,A_2,B_1,B_2\in\mathcal{E}$.
Define, for $w\in \mathcal{G}^0$,
\[
q_w \;:=\; p_w \;-\; \sum_{\{e\in \mathcal{G}^1:\ s(e)=w\}} s_e s_e^* .
\]
Then the following hold:
\begin{enumerate}
    \item \(p_{A_1} \left( \sum\limits_{e \mid s(e) \in A_2 \cap Y} s_e s_e^* \right) = p_{A_1 \cap Y} \left( \sum\limits_{e \mid s(e) \in A_2 \cap Y} s_e s_e^* \right)\).
    \item \(\left( \sum\limits_{e \mid s(e) \in A_1 \cap Y} s_e s_e^* \right) p_{A_2} = \left( \sum\limits_{e \mid s(e) \in A_1 \cap Y} s_e s_e^* \right) p_{A_2 \cap Y}\).
    \item \(p_{A_1} \left( \sum\limits_{w \mid w \in B_2 \cap Y} q_w \right) = p_{A_1 \cap Y} \left( \sum\limits_{w \mid w \in B_2 \cap Y} q_w \right)\).
    \item \(\left( \sum\limits_{w \mid w \in B_1 \cap Y} q_w \right) p_{A_2} = \left( \sum\limits_{w \mid w \in B_1 \cap Y} q_w \right) p_{A_2 \cap Y}\).
    \item \(\left( \sum\limits_{e \mid s(e) \in A_1 \cap Y} s_e s_e^* \right) \left( \sum\limits_{w \in B_2 \cap Y} q_w \right) = 0 = \left( \sum\limits_{w \in B_1 \cap Y} q_w \right) \left( \sum\limits_{e \mid s(e) \in A_2 \cap Y} s_e s_e^* \right)\).
    \item \(q_w q_v = \delta_{wv} q_w\), and consequently, \(\left( \sum\limits_{w \in B_1 \cap Y} q_w \right) \left( \sum\limits_{w \in B_2 \cap Y} q_w \right) = \sum\limits_{w \in B_1 \cap B_2 \cap Y} q_w\).
\end{enumerate}
\end{lemma}

\begin{proof}
  
The first item holds because, for all \(e \in E^1\) such that \(s(e) \in A_2 \cap Y\), we have  $
p_{A_1} s_e s_e^* = p_{A_1 \cap Y} s_e s_e^*$, which equals $ s_e s_e^*$ if $s(e) \in A_1$, and equals 0 if $s(e) \notin A_1$.
The second item follows analogously. The same reasoning applies to the third item:  
\begin{align*}
   p_{A_1 \cap Y} \left( \sum\limits_{w \mid w \in B_2 \cap Y } q_w\right) & =  \sum\limits_{w \mid w \in B_2 \cap Y } 
 \left( p_{A_1 \cap Y} p_w - \sum\limits_{e \mid s(e) = w} p_{A_1 \cap Y} s_e s_e^* \right)  \\
 & = \sum\limits_{w \mid w \in B_2 \cap Y } 
 \left( p_{A_1} p_w - \sum\limits_{e \mid s(e) = w} p_{A_1} s_e s_e^* \right) = p_{A_1} \left( \sum\limits_{w \mid w \in B_2 \cap Y } q_w\right). 
\end{align*}  
The fourth item follows the same logic as the third. The left-hand side of the fifth item’s equality holds because

\begin{align*}
   & \left( \sum_{s(e) \in A_1 \cap Y} s_e s_e^* \right) 
   \left( \sum_{w \in B_2 \cap Y} q_w \right) = 
   \left( \sum_{s(e) \in A_1 \cap Y} s_e s_e^* \right) 
   \left( \sum_{w \in B_2 \cap Y} \left( p_w - \sum_{s(f) = w} s_f s_f^* \right) \right) \\
    & = \sum_{w \in B_2 \cap Y} \left[ 
    \left( \sum_{s(e) \in A_1 \cap Y} s_e s_e^* \right) p_w - 
    \left( \sum_{s(e) \in A_1 \cap Y} s_e s_e^* \right) 
    \left( \sum_{s(f) = w} s_f s_f^* \right) \right] \\
    & = \sum_{w \in B_2 \cap Y} \left[ 
    \left( \sum_{s(e) = w \in A_1 \cap Y} s_e s_e^* \right) - 
    \left( \sum_{s(e) = w \in A_1 \cap Y} s_e s_e^* \right) 
    \right] = 0.
\end{align*}
and the right-hand side of the fifth item’s  equality can be proved by the same computation.  Finally, the last item follows from:

\begin{align*}
    q_w q_v & = \left(p_w - \sum_{s(e) = w} s_e s_e^* \right) 
    \left(p_v - \sum_{s(f) = v} s_f s_f^* \right) \\
    & = p_w p_v - \sum_{s(f) = v} p_w s_f s_f^* - 
    \sum_{s(e) = w} s_e s_e^* p_v + 
    \sum_{s(e) = w} \sum_{s(f) = v} s_e s_e^* p_w p_v s_f s_f^* \\
    & = \delta_{wv} \left(p_w - \sum_{s(e) = w} s_e s_e^* \right) = \delta_{wv}q_w.
\end{align*}

\end{proof}

We are now able to prove that $\phi$ has a left inverse.

\begin{proposition}\label{isomorphismrelative}
Let $(\mathcal{G},X)$ be a relative ultragraph and set
$Y:=\mathrm{Reg}(\mathcal{G})\setminus X$. If $r(e)\cap Y$ is finite for every
$e\in\mathcal{G}^1$, then the $*$-homomorphism
$\phi: C^*(\mathcal{G},X)\to C^*(\mathcal{G}_X)$ from
Proposition~\ref{esfriou, vai dar tainha} is injective.
\end{proposition}
\begin{proof}
Following the notation of Lemma~\ref{contas1}, define, for all
$w \in \mathrm{Reg}(\mathcal{G})=\mathrm{Reg}(\mathcal{G}_X)$,
\[
q_w \;:=\; p_w \;-\; \sum_{\substack{e \in \mathcal{G}^1 \\ s(e)=w}} s_e s_e^* .
\]
Let $Z \in \mathcal{E}_X^0$. By Corollary~\ref{decomposition2}, we can write $Z$ uniquely as
$Z = A \cup (B \cap Y)'$ where $A,B \in \mathcal{E}^0$ and $(B \cap Y)'$ is finite. We define $\psi$
on the generators of $C^*(\mathcal{G}_X)$ as follows:
\[
\psi(P_Z)
\;=\;
p_A - p_{A \cap Y}
\;+\;
\sum_{\substack{e \in \mathcal{G}^1 \\ s(e) \in A \cap Y}} s_e s_e^*
\;+\;
\sum_{\substack{w \in \mathcal{G}^0 \\ w \in B \cap Y}} q_w,
\]
and
\[
\psi(S_f)
=
\begin{cases}
\displaystyle
s_f \!\left(
p_{r(f)}
-
p_{r(f)\cap Y}
+
\sum_{\substack{g \in \mathcal{G}^1 \\ s(g) \in r(f)\cap Y}}
s_g s_g^*
\right),
& \text{if } f \in \mathcal{G}^1, \\[10pt]
\displaystyle
s_e \!\left(
p_{r(e)\cap Y}
-
\sum_{\substack{g \in \mathcal{G}^1 \\ s(g) \in r(e)\cap Y}}
s_g s_g^*
\right),
& \text{if } f=e' \in C.
\end{cases}
\]
Sums over the empty set are taken to be $0$. The sums written above are finite by the hypothesis.

We claim that $\psi$ is well defined and that $\psi \circ \phi = \mathrm{id}$. To see that $\psi$ is
well defined, we use the universal property of $C^*(\mathcal{G}_X)$; some computations are necessary:

\begin{enumerate}
\item We verify that $\psi(P_\emptyset)=0$, that
$\psi(P_{Z_1 \cap Z_2}) = \psi(P_{Z_1}) \psi(P_{Z_2})$, and that
$\psi(P_{Z_1\cup Z_2})= \psi(P_{Z_1})+\psi(P_{Z_2})-\psi(P_{Z_1\cap Z_2})$ for all
$Z_1, Z_2 \in \mathcal{E}_X^0$.

Of course $\psi(P_{\emptyset}) = p_{\emptyset} = 0$. Let us now prove that
$\psi(P_{Z_1 \cap Z_2}) = \psi(P_{Z_1}) \psi(P_{Z_2})$ for all $Z_1, Z_2 \in \mathcal{E}_X^0$.
Write
\[
Z_1 = A_1 \cup (B_1 \cap Y)'
\quad\text{and}\quad
Z_2 = A_2 \cup (B_2 \cap Y)'
\]
as in Corollary~\ref{decomposition2}. Note that $\psi(P_{Z_1}) \psi(P_{Z_2})$ is the product of
\[
\left(
p_{A_1} - p_{A_1 \cap Y}
+ \sum_{\substack{e \in \mathcal{G}^1 \\ s(e) \in A_1 \cap Y}} s_e s_e^*
+ \sum_{\substack{w \in \mathcal{G}^0 \\ w \in B_1 \cap Y}} q_w
\right)
\]
and
\[
\left(
p_{A_2} - p_{A_2 \cap Y}
+ \sum_{\substack{e \in \mathcal{G}^1 \\ s(e) \in A_2 \cap Y}} s_e s_e^*
+ \sum_{\substack{w \in \mathcal{G}^0 \\ w \in B_2 \cap Y}} q_w
\right).
\]
Using the equalities provided in Lemma~\ref{contas1}, most of the products are $0$ or cancel with
another term, so that
\[
\psi(P_{Z_1})\psi(P_{Z_2})
=
p_{A_1 \cap A_2}
-
p_{A_1 \cap A_2 \cap Y}
+
\sum_{\substack{e \in \mathcal{G}^1 \\ s(e) \in A_1 \cap A_2 \cap Y}} s_e s_e^*
+
\sum_{w \in B_1 \cap B_2 \cap Y} q_w.
\]
Thus we have the desired result; indeed, the right-hand side is exactly the definition of
$\psi(P_{Z_1 \cap Z_2})$, because
$Z_1 \cap Z_2 = (A_1 \cap A_2) \cup (B_1 \cap B_2 \cap Y)'$.

To finalize item~(1), we still need to show that
$\psi(P_{Z_1}) + \psi(P_{Z_2}) - \psi(P_{Z_1 \cap Z_2}) = \psi(P_{Z_1 \cup Z_2})$.
This is similar to what we did above, and we leave it to the interested reader.

\item In this item we show that $\psi(S_f)^*\psi(S_f) = \psi(P_{r(f)})$. If $f \in \mathcal{G}^1$,
then $\psi(S_f)^* \psi(S_f)$ equals
\[
\left(
s_f - s_f p_{r(f) \cap Y}
+ \sum_{\substack{g \in \mathcal{G}^1 \\ s(g) \in r(f) \cap Y}} s_f s_g s_g^*
\right)^{\!*}
\left(
s_f - s_f p_{r(f) \cap Y}
+ \sum_{\substack{g \in \mathcal{G}^1 \\ s(g) \in r(f) \cap Y}} s_f s_g s_g^*
\right).
\]
Using that, for all $g$ with $s(g) \in r(f) \cap Y$, we have
$p_{r(f) \cap Y}s_g s_g^* = s_g s_g^*$ and $p_{r(f)}s_g s_g^* = s_g s_g^*$, we conclude that
\[
\psi(S_f)^* \psi(S_f)
=
p_{r(f)} - p_{r(f) \cap Y}
+ \sum_{\substack{g \in \mathcal{G}^1 \\ s(g) \in r(f) \cap Y}} s_g s_g^*
=
\psi(P_{r(f)}),
\]
where the last equality holds because $r(f) \subseteq \mathcal{G}^0$.

The remaining case is when $f = e'$ for some $e \in \mathcal{G}^1$. In this situation,
$r(f) = r(e') = \emptyset \cup (r(e) \cap Y)'$, and thus
\[
\psi(P_{r(f)})
=
\psi(P_{r(e')})
=
\sum_{\substack{w \in \mathcal{G}^0 \\ w \in r(e) \cap Y}} q_w
=
p_{r(e) \cap Y}
-
\sum_{\substack{g \in \mathcal{G}^1 \\ s(g) \in r(e) \cap Y}} s_g s_g^*,
\]
where the last equality holds because the right-hand side is the left-hand side after reordering
the sum. Moreover, a computation shows that
\begin{align*}
\psi(S_f)^*\psi(S_f)
&=
\left(
s_e  p_{r(e) \cap Y}
-
\sum_{\substack{g \in \mathcal{G}^1 \\ s(g) \in r(e) \cap Y}} s_e s_g s_g^*
\right)^{\!*}
\left(
s_e  p_{r(e) \cap Y}
-
\sum_{\substack{g \in \mathcal{G}^1 \\ s(g) \in r(e) \cap Y}} s_e s_g s_g^*
\right) \\
&=
\left(
p_{r(e) \cap Y} s_e^*
-
\sum_{\substack{g \in \mathcal{G}^1 \\ s(g) \in r(e) \cap Y}} s_g s_g^* s_e^*
\right)
\left(
s_e  p_{r(e) \cap Y}
-
\sum_{\substack{g \in \mathcal{G}^1 \\ s(g) \in r(e) \cap Y}} s_e s_g s_g^*
\right) \\
&=
p_{r(e) \cap Y}
-
\sum_{\substack{g \in \mathcal{G}^1 \\ s(g) \in r(e) \cap Y}} s_g s_g^*
-
\sum_{\substack{g \in \mathcal{G}^1 \\ s(g) \in r(e) \cap Y}} s_g s_g^*
+
\sum_{\substack{g \in \mathcal{G}^1 \\ s(g) \in r(e) \cap Y}} s_g s_g^* \\
&=
p_{r(e) \cap Y}
-
\sum_{\substack{g \in \mathcal{G}^1 \\ s(g) \in r(e) \cap Y}} s_g s_g^*
=
\psi(P_{r(f)}).
\end{align*}

\item We need to prove that, for all $f \in \mathcal{G}_X^1$,
\[
\psi(P_{s(f)}) \,\psi(S_f)\, \psi(S_f)^* \;=\; \psi(S_f)\, \psi(S_f)^*.
\]
To start, let us compute $\psi(S_f) \psi(S_f)^*$; this computation will be useful also in the next
item. There are two possible situations:

\begin{itemize}
\item \textbf{Case 1}. $f \in \mathcal{G}^1$. Using the definition of $\psi$ and performing some computations, we obtain
\begin{align*}
\psi(S_f)\psi(S_f)^*
&= s_f s_f^* - s_f\, p_{r(f)\cap Y}\, s_f^*
  + \sum_{\substack{g\in\mathcal{G}^1\\ s(g)\in r(f)\cap Y}}
    s_f s_g s_g^* s_f^* \\[2pt]
&= s_f\!\left(
    s_f^* - p_{r(f)\cap Y}\, s_f^*
    + \sum_{\substack{g\in\mathcal{G}^1\\ s(g)\in r(f)\cap Y}}
      s_g s_g^* s_f^*
  \right).
\end{align*}

\item \textbf{Case 2}. $f = e' \in C$. Using the definition of $\psi$ and performing some computations, we obtain
\[
\psi(S_{e'})\psi(S_{e'})^*
=
s_e p_{r(e) \cap Y} s_e^*
-
\sum_{\substack{g \in \mathcal{G}^1 \\ s(g) \in r(e) \cap Y}} s_e s_g s_g^* s_e^*
=
s_e \!\left(
p_{r(e) \cap Y} s_e^*
-
\sum_{\substack{g \in \mathcal{G}^1 \\ s(g) \in r(e) \cap Y}} s_g s_g^* s_e^*
\right)\!.
\]
\end{itemize}

In both cases, we now compute the product $\psi(P_{s(f)})\,\psi(S_f)\,\psi(S_f)^*$. Note that, for all $f \in \mathcal{G}_X^1$,
$s(f) \in \mathcal{G}^0$; moreover,
\[
\psi(P_{s(f)}) \;=\;
\begin{cases}
p_{s(f)}, & \text{if } s(f) \notin Y, \\[2pt]
\displaystyle \sum_{\substack{e \in \mathcal{G}^1 \\ s(e) = s(f)}} s_e s_e^*, & \text{if } s(f) \in Y.
\end{cases}
\]
Thus, in each case above we need to consider two subcases. The reader can check all these
possibilities by direct computations; they are very similar and use the facts that
$p_{s(f)} s_f = s_f$ and
$\sum_{\substack{e \in \mathcal{G}^1 \\ s(e) = s(f)}} s_e s_e^*\, s_f = s_f$.

The three cases above can be used to check that the $\psi(S_f)$ are partial isometries and that
$\psi(S_f) \psi(S_f)^* \, \psi(S_g) \psi(S_g)^* = 0$ if $f \neq g$. We leave this to the reader.

\item In this item we show that, for all $v \in \mathrm{Reg}(\mathcal{G}_X)=\mathrm{Reg}(\mathcal{G})$,
\[
\psi(P_v) \;=\; \sum_{\substack{f \in \mathcal{G}_X^1 \\ s(f) = v}} \psi(S_f)\, \psi(S_f)^*.
\]
For notational convenience, set $O = \{\,f \in \mathcal{G}^1 \mid r(f) \cap Y = \emptyset\,\}$ and
$Q = \{\,f \in \mathcal{G}^1 \mid r(f) \cap Y \neq \emptyset\,\}$ only for this item. Note that
$\mathcal{G}^1 = O \cup Q$, thus
\[
\mathcal{G}_X^1 = O \cup Q \cup C.
\]
Using the computations made in item~(3), we note that
\[
\sum_{\substack{f \in O \\ s(f) = v}} \psi(S_f) \psi(S_f)^*
=
\sum_{\substack{f \in O \\ s(f) = v}} s_f s_f^*,
\]
\[
\sum_{\substack{f \in Q \\ s(f) = v}} \psi(S_f) \psi(S_f)^*
=
\sum_{\substack{f \in Q \\ s(f) = v}}
\Big(
s_f s_f^*
-
s_f p_{r(f) \cap Y} s_f^*
+
\sum_{\substack{g \in \mathcal{G}^1 \\ s(g) \in r(f) \cap Y}}
s_f s_g s_g^* s_f^*
\Big),
\]
and
\[
\sum_{\substack{f' \in C \\ s(f') = v}} \psi(S_{f'}) \psi(S_{f'})^*
=
\sum_{\substack{f' \in C \\ s(f') = v}}
\Big(
s_f p_{r(f) \cap Y} s_f^*
-
\sum_{\substack{g \in \mathcal{G}^1 \\ s(g) \in r(f) \cap Y}}
s_f s_g s_g^* s_f^*
\Big).
\]
Using that $s(f') = s(f)$ for all $f \in \mathcal{G}^1$ with $r(f) \cap Y \neq \emptyset$, we can sum
the three equalities above to obtain
\begin{align*}
\sum_{\substack{f \in \mathcal{G}_X^1 \\ s(f) = v}} \psi(S_f) \psi(S_f)^*
&=
\sum_{\substack{f \in O \\ s(f) = v}} s_f s_f^*
\;+\;
\sum_{\substack{f \in Q \\ s(f) = v}} s_f s_f^* \\
&=
\sum_{\substack{f \in \mathcal{G}^1 \\ s(f) = v}} s_f s_f^*.
\end{align*}
Since $v \in \mathcal{G}^0$ and $\{v\} = \{v\} \cup \emptyset$, we have two options: $v \notin Y$ or
$v \in Y$. In the first case, $\psi(P_v) = p_v$ and, since $v \notin Y$ (i.e., $v \in X$),
$p_v = \sum_{\substack{f \in \mathcal{G}^1 \\ s(f) = v}} s_f s_f^*$. In the second case, when
$v \in Y$, we have $\psi(P_v) = \sum_{\substack{f \in \mathcal{G}^1 \\ s(f) = v}} s_f s_f^*$.
In both cases we obtain the desired result. With these four items, we conclude that $\psi$ extends
to a well-defined homomorphism.

To conclude the theorem, it remains to check that $\psi \circ \phi = \mathrm{id}_{C^*(\mathcal{G},X)}$.
First, for all $v \in \mathcal{G}^0$ we have $\psi \circ \phi(p_v) = p_v$; this is immediate if
$v \notin Y$, because in this case $\psi \circ \phi(p_v) = \psi(P_v) = p_v$. If $v \in Y$, then
\[
\psi \circ \phi(p_v)
=
\psi(P_v + P_{v'})
=
\sum_{\substack{e \in \mathcal{G}^1 \\ s(e) = v}} s_e s_e^*
\;+\;
\left( p_v - \sum_{\substack{e \in \mathcal{G}^1 \\ s(e) = v}} s_e s_e^* \right)
=
p_v.
\]
Next, for all $e \in \mathcal{G}^1$,
\begin{align*}
\psi \circ \phi(p_{r(e)})
&= \psi\big( P_{r(e)} + P_{(r(e) \cap Y)'} \big) \\
&= p_{r(e)} - p_{r(e) \cap Y}
+ \sum_{\substack{g \in \mathcal{G}^1 \\ s(g) \in r(e) \cap Y}} s_g s_g^*
+ \sum_{w \in r(e) \cap Y} q_w \\
&= p_{r(e)} - p_{r(e) \cap Y}
+ \sum_{\substack{g \in \mathcal{G}^1 \\ s(g) \in r(e) \cap Y}} s_g s_g^*
+ p_{r(e) \cap Y}
- \sum_{\substack{g \in \mathcal{G}^1 \\ s(g) \in r(e) \cap Y}} s_g s_g^* \\
&= p_{r(e)}.
\end{align*}
Finally, let $e \in \mathcal{G}^1$. If $r(e) \cap Y = \emptyset$, then
\[
\psi \circ \phi(s_e) = \psi(S_e) = s_e.
\]
If $r(e) \cap Y \neq \emptyset$, then
\begin{align*}
\psi \circ \phi(s_e)
&= \psi(S_e + S_{e'}) \\
&= s_e - s_e p_{r(e) \cap Y}
+ \sum_{\substack{g \in \mathcal{G}^1 \\ s(g) \in r(e) \cap Y}} s_e s_g s_g^*
\;+\;
s_e p_{r(e) \cap Y}
- \sum_{\substack{g \in \mathcal{G}^1 \\ s(g) \in r(e) \cap Y}} s_e s_g s_g^* \\
&= s_e.
\end{align*}
Since $\psi \circ \phi = \mathrm{id}_{C^*(\mathcal{G},X)}$ on the set of generators of
$C^*(\mathcal{G},X)$, the result follows.
\end{enumerate}
\end{proof}

For future reference and clarity, we record the following theorem, which is an immediate
consequence of Propositions~\ref{surjective homomorphism} and \ref{isomorphismrelative}.
This result generalizes \cite[Theorem~3.7]{Tom}.

\begin{theorem}\label{teo iso algebras ultragrafos e relativas}
Let $(\mathcal{G},X)$ be a relative ultragraph, and set
$Y := \mathrm{Reg}(\mathcal{G}) \setminus X$. If $r(e) \cap Y$ is finite for all
$e \in \mathcal{G}^1$, then the $*$-homomorphism
\[
\phi: C^*(\mathcal{G},X) \to C^*(\mathcal{G}_X)
\]
from Proposition~\ref{esfriou, vai dar tainha} is a $*$-isomorphism, with inverse given by the
$*$-homomorphism $\psi$ of Proposition~\ref{isomorphismrelative}.
\end{theorem}

\section{Uniqueness theorems for relative ultragraph $C^*$-algebras.}

In this section we establish uniqueness theorems for relative ultragraph $C^*$-algebras.
To this end, we analyze the relationship between cycles in $\mathcal{G}$ and in the associated
ultragraph $\mathcal{G}_X$, and introduce the notion of \emph{Relative Condition~(L)} adapted to
the relative setting. We then combine these ingredients with classical results for ultragraph
$C^*$-algebras and the isomorphism of
Theorem~\ref{teo iso algebras ultragrafos e relativas} to obtain our uniqueness theorems.

Recall that a \emph{cycle} in an ultragraph $\mathcal{G}$ is a finite path
$\alpha=\alpha_1 \cdots \alpha_n$, where each $\alpha_i$ is an edge, such that
$s(\alpha_{i+1}) \in r(\alpha_i)$ for every $i \in \{1,\dots,n-1\}$ and
$s(\alpha_1) \in r(\alpha_n)$. Following \cite{Tom2}, an \emph{exit} for a cycle $\alpha$ is
either a sink in $r(\alpha_i)$ for some $i \in \{1,\dots,n\}$, or an edge $e$ satisfying one of
the following:
\begin{itemize}
    \item $s(e) \in r(\alpha_i)$ and $e \neq \alpha_{i+1}$ for some $i \in \{1,\dots,n-1\}$;
    \item $s(e) \in r(\alpha_n)$ and $e \neq \alpha_1$.
\end{itemize}
An ultragraph is said to satisfy \textbf{Condition~(L)} if every cycle has an exit.

\begin{remark}\label{Same cycles}
Let $(\mathcal{G},X)$ be a relative ultragraph, and let $\mathcal{G}_X$ be the associated
ultragraph as in Definition~\ref{ultragrap GX}. Note that all vertices added to $\mathcal{G}$ to construct
$\mathcal{G}_X$ are sinks. Consequently, the sets of cycles in $\mathcal{G}$ and in
$\mathcal{G}_X$ coincide.
\end{remark}

It follows from Remark~\ref{Same cycles} that if $\mathcal{G}$ satisfies Condition~(L), then so
does $\mathcal{G}_X$. The converse, however, does not hold, as illustrated in the following
example:

\begin{example}
Let $\mathcal{G}$ be the graph with two vertices and two edges as follows:

\vspace{1cm}
\centerline{%
\setlength{\unitlength}{2cm}%
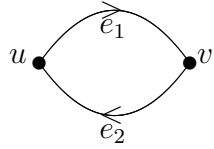
\begin{picture}(0,0)
\put(0,0){\circle*{0.08}}
\put(-0.2,0){$u$}
\put(0,0){\qbezier(0,0)(0.5,0.7)(1,0)}
\put(1,0){\circle*{0.08}}
\put(1.05,0){$v$}
\put(0.4,0.3){$>$}
\put(0.4,0.2){$e_1$}
\put(0,0){\qbezier(0,0)(0.5,-0.7)(1,0)}
\put(0.4,-0.5){$e_2$}
\put(0.4,-0.4){$<$}
\end{picture}}
\vspace{1.5cm}

Note that $\mathrm{Reg}(\mathcal{G})=\{u,v\}$ and set $X=\{u\}$. Following
Definition~\ref{ultragrap GX}, we obtain the associated ultragraph $\mathcal{G}_X$:

\vspace{1.5cm}
\centerline{%
\setlength{\unitlength}{2cm}%
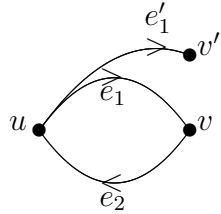
\begin{picture}(0,0)
\put(0,0){\circle*{0.08}}
\put(-0.2,0){$u$}
\put(0,0){\qbezier(0,0)(0.5,0.7)(1,0)}
\put(1,0){\circle*{0.08}}
\put(1.05,0){$v$}
\put(0.4,0.3){$>$}
\put(0.4,0.2){$e_1$}
\put(0,0){\qbezier(0,0)(0.5,-0.7)(1,0)}
\put(0.4,-0.5){$e_2$}
\put(0.4,-0.4){$<$}
\put(0,0){\qbezier(0,0)(0.5,0.7)(1,0.5)}
\put(1,0.5){\circle*{0.08}}
\put(1.05,0.5){$v'$}
\put(0.7,0.7){$e_1'$}
\put(0.7,0.5){$>$}
\end{picture}}
\vspace{1.5cm}

Thus $\mathcal{G}$ does not satisfy Condition~(L), whereas the ultragraph $\mathcal{G}_X$
does satisfy Condition~(L).
\end{example}

The notion of Relative Condition~(L) for graphs was introduced in \cite{danielcristobal}.
We now extend this concept to the setting of ultragraphs.

\begin{definition}
Let $(\mathcal{G}, X)$ be a relative ultragraph and set $Y := \mathrm{Reg}(\mathcal{G}) \setminus X$.
We say that $\mathcal{G}$ satisfies the \emph{Relative Condition~(L)} if, for every cycle without
exits in $\mathcal{G}$, there exists an edge $e$ in the cycle such that
$r(e)\cap Y \neq \emptyset$.
\end{definition}

\begin{remark}
If $\alpha=\alpha_1\cdots \alpha_n$ is a cycle without exits in an ultragraph $\mathcal{G}$,
then for each $i$ the range $r(\alpha_i)$ is a singleton. Consequently, in the definition above
we may equivalently write $r(e)\in Y$ instead of $r(e)\cap Y\neq \emptyset$.
\end{remark}

Next, we relate exits in a relative ultragraph $(\mathcal{G},X)$ with exits in the ultragraph
$\mathcal{G}_X$.

\begin{lemma}\label{lemma relative condition L}
Let $(\mathcal{G}, X)$ be a relative ultragraph, let $Y=\mathrm{Reg}(\mathcal{G})\setminus X$,
and let $\mathcal{G}_X$ be the associated ultragraph as in Definition~\ref{ultragrap GX}.
\begin{enumerate}
    \item A cycle $\alpha=\alpha_1\cdots \alpha_n$ has no exit in $\mathcal{G}_X$ if and only if
    it has no exit in $\mathcal{G}$ and $r(\alpha_i)\cap Y\neq \emptyset$ for each
    $i\in\{1,\dots,n\}$.
    \item The ultragraph $\mathcal{G}_X$ satisfies Condition~(L) if and only if $\mathcal{G}$
    satisfies the Relative Condition~(L).
\end{enumerate}
\end{lemma}

\begin{proof}
The proof of the first item follows directly from the definition of the ultragraph $\mathcal{G}_X$,
and we leave it to the reader. The second item follows immediately from the first.
\end{proof}

Before proving the first uniqueness theorem, we present an auxiliary result.

\begin{lemma}\label{Aux2condL}
Let $(\G,X)$ be a relative ultragraph and set $Y=\mathrm{Reg}(\G)\setminus X$. Suppose that
$r(e)\cap Y$ is finite for all edges $e \in \G^1$. If $Z=A \cup (B \cap Y)'$, as in
Corollary~\ref{decomposition2}, then $p_A - p_{A \cap Y}$, $s_e s_e^*$ with $e \in \G^1$, and
$q_v$ with $v \in B \cap Y$, are mutually orthogonal projections.
\end{lemma}

\begin{proof}
By a straightforward computation one sees that $p_A - p_{A \cap Y}$ and $s_e s_e^*$ are projections.
Of course $s_e s_e^* s_f s_f^* = \delta_{ef}\, s_e s_e^*$. We proved that
$q_v q_w = \delta_{vw}\, q_v$ in Lemma~\ref{contas1}.

Note that, for all $e \in \G^1$ with $s(e) \in A \cap Y$, we have
$p_A s_e s_e^* = s_e s_e^*$ and $p_{A \cap Y} s_e s_e^* = s_e s_e^*$, showing that
\[
\bigl( p_A - p_{A \cap Y} \bigr) s_e s_e^* = 0.
\]
Also, let $v \in B \cap Y$. Then
\[
p_A \!\left(p_v - \sum_{\substack{f \in \G^1 \\ s(f) = v}} s_f s_f^* \right)
=
\begin{cases}
p_v - \displaystyle\sum_{\substack{f \in \G^1 \\ s(f) = v}} s_f s_f^*, & \text{if } v \in A,\\[10pt]
0, & \text{if } v \notin A,
\end{cases}
=
p_{A \cap Y} \!\left(p_v - \sum_{\substack{f \in \G^1 \\ s(f) = v}} s_f s_f^* \right),
\]
showing that
\[
(p_A - p_{A \cap Y})\, q_v
=
(p_A - p_{A \cap Y}) \!\left(p_v - \sum_{\substack{f \in \G^1 \\ s(f) = v}} s_f s_f^* \right)
= 0.
\]
Finally, if $e \in \G^1$ is such that $s(e) \in A \cap Y$ and $v \in B \cap Y$, then
\[
s_e s_e^* \!\left(p_v - \sum_{\substack{f \in \G^1 \\ s(f) = v}} s_f s_f^* \right) = 0.
\]
Indeed, if $v \neq s(e)$ this is clear. If $v = s(e)$, the left-hand side equals
\[
s_e s_e^* - \sum_{\substack{f \in \G^1 \\ s(f) = v}} s_e s_e^* s_f s_f^*
= s_e s_e^* - s_e s_e^* = 0.
\]
\end{proof}

The following remark will be used in the next theorem.

\begin{remark}\label{Aux1condL}
Let $\mathcal{A}$ and $\mathcal{B}$ be $C^*$-algebras and let
$\Phi:\mathcal{A}\to\mathcal{B}$ be a $*$-homomorphism. If $p_1,\ldots,p_N$ are mutually
orthogonal projections in $\mathcal{A}$ and $\Phi(p_i)\neq 0$ for at least one
$i \in \{1,\ldots,N\}$, then
\[
0 \neq \Phi(p_i) = \Phi(p_i)\,\Phi(p_1 + p_2 + \cdots + p_N),
\]
and so $\Phi(p_1 + p_2 + \cdots + p_N) \neq 0$.
\end{remark}

\begin{theorem}\label{Relative condition (L) implies injective}
Let $(\mathcal{G},X)$ be a relative ultragraph and set $Y=\mathrm{Reg}(\mathcal{G})\setminus X$,
with $Y\cap r(e)$ finite for each edge $e\in\mathcal{G}^1$. Let $\mathcal{B}$ be a $C^*$-algebra
and let $\Phi:C^*(\mathcal{G},X)\to \mathcal{B}$ be a $*$-homomorphism. If
\begin{enumerate}
    \item $\Phi(p_v)\neq 0$ for all $v\notin Y$;
    \item for each $e\in\mathcal{G}^1$ with $s(e)\in Y$ we have $\Phi(s_e s_e^*)\neq 0$;
    \item for all $v\in Y$, $\Phi(q_v)\neq 0$ \ (where $q_v$ is as in Lemma~\ref{contas1}),
\end{enumerate}
then $\Phi\circ\psi(P_Z)\neq 0$ for every nonempty $Z\in\mathcal{E}_X$, where $\psi$ is the
isomorphism obtained in Theorem~\ref{teo iso algebras ultragrafos e relativas}. Moreover, if
$\mathcal{G}$ satisfies the Relative Condition~(L), then $\Phi$ is injective.
\end{theorem}

\begin{proof}
Using the definition of $\psi$ and writing $Z = A \cup (B \cap Y)'$, we obtain
\[
\Phi \circ \psi(P_Z)
=
\Phi\!\left(
p_A - p_{A \cap Y}
+ \sum_{\substack{e \in \mathcal{G}^1 \\ s(e) \in A \cap Y}} s_e s_e^*
+ \sum_{v \in B \cap Y} q_v
\right).
\]

We will show that $\Phi \circ \psi(P_Z) \neq 0$ whenever $Z$ is non-empty.
\begin{itemize}
    \item If $A \neq \emptyset$ and $A \cap Y = \emptyset$, then $p_A - p_{A \cap Y} = p_A$. Note that $\Phi(p_A) \neq 0$, because if $v \in A$ then
    \[
    0 \neq \Phi(p_v) = \Phi(p_v p_A) = \Phi(p_v)\,\Phi(p_A),
    \]
    hence $\Phi(p_A) \neq 0$, and thus $\Phi(p_A - p_{A \cap Y}) = \Phi(p_A) \neq 0$.
    \item If $A \neq \emptyset$ and $A \cap Y \neq \emptyset$, choose $v \in A \cap Y$. Since $Y \subseteq \mathrm{Reg}(\mathcal{G})$, pick $e \in \mathcal{G}^1$ with $s(e) = v \in A \cap Y$; by hypothesis, $\Phi(s_e s_e^*) \neq 0$.
    \item If $A = \emptyset$, then $B \cap Y \neq \emptyset$, so there exists $v \in B \cap Y$ with $\Phi(q_v) \neq 0$.
\end{itemize}

In each of the three situations, Lemma~\ref{Aux2condL} and Remark~\ref{Aux1condL} imply that
$\Phi \circ \psi(P_Z) \neq 0$ for every non-empty $Z \in \mathcal{E}_X$. Moreover, if
$\mathcal{G}$ satisfies the Relative Condition~(L), then by
Lemma~\ref{lemma relative condition L} we have that $\mathcal{G}_X$ satisfies Condition~(L).
Hence the hypotheses of \cite[Theorem~6.7]{Tom2} are satisfied, showing that $\Phi \circ \psi$
is injective. Since $\psi$ is an isomorphism, it follows that $\Phi$ is injective.
\end{proof}

The next theorem fully characterizes the injectivity of $*$-homomorphisms of relative ultragraph
$C^*$-algebras.

\begin{theorem}
Let $(\mathcal{G},X)$ be a relative ultragraph and set $Y=\mathrm{Reg}(\mathcal{G})\setminus X$,
with $Y\cap r(e)$ finite for each edge $e\in\mathcal{G}^1$. Let $\mathcal{B}$ be a $C^*$-algebra
and let $\Phi:C^*(\mathcal{G},X)\to \mathcal{B}$ be a $*$-homomorphism. Then $\Phi$ is injective
if and only if the following conditions hold:
\begin{enumerate}
    \item $\Phi(p_v)\neq 0$ for all $v\notin Y$;
    \item for each $e\in\mathcal{G}^1$ with $s(e)\in Y$ we have $\Phi(s_e s_e^*)\neq 0$;
    \item for all $v\in Y$, $\Phi(q_v)\neq 0$ \ (where $q_v$ is as in Lemma~\ref{contas1});
    \item for any simple cycle $\alpha=\alpha_1\ldots \alpha_n$ with no exit in $\mathcal{G}$ and
    such that $r(\alpha_i)\notin Y$ for all $i=1,\ldots,n$, the spectrum of $\Phi(s_\alpha)$
    contains the unit circle.
\end{enumerate}
\end{theorem}

\begin{proof}
Suppose that $\Phi$ is injective. Since $p_A$ (for all $A \in \mathcal{E}$) and $s_e s_e^*$ (for
$e \in \mathcal{G}^1$) are nonzero elements, items~(1) and~(2) follow. Item~(3) follows because
$\Phi \circ \psi$ is injective (where $\psi$ is as in
Theorem~\ref{teo iso algebras ultragrafos e relativas}) and $P_{v'} \neq 0$, hence
\[
0 \neq \Phi \circ \psi(P_{v'}) = \Phi(q_v).
\]
Moreover, since $\Phi \circ \psi$ is injective, we may apply \cite[Theorem~7.4]{DDH} to deduce
that the \emph{spectrum} of $\Phi \circ \psi(S_{\alpha})$ contains the unit circle whenever
$\alpha$ is a simple cycle with no exit in $\mathcal{G}$ and $r(\alpha_i) \notin Y$ for all $i$.
By the definition of $\psi$,
\[
\Phi \circ \psi(S_{\alpha})
= \Phi \circ \psi (S_{\alpha_1} \cdots S_{\alpha_n})
= \Phi(s_{\alpha_1} \cdots s_{\alpha_n})
= \Phi(s_{\alpha}),
\]
so item~(4) holds.

\smallskip
For the converse, suppose that the four conditions hold. By
Theorem~\ref{Relative condition (L) implies injective}, the first three conditions ensure that
$\Phi \circ \psi (P_Z) \neq 0$ for all $Z \in \mathcal{E}_X$. If $\alpha$ is a cycle with no
exit in $\mathcal{G}_X$, then by item~(1) of Lemma~\ref{lemma relative condition L} and by~(4),
the spectrum of $\Phi \circ \psi(S_{\alpha}) = \Phi(s_{\alpha})$ contains the unit circle.
Therefore, by \cite[Theorem~7.4]{DDH}, $\Phi \circ \psi$ is injective, and since $\psi$ is an
isomorphism, it follows that $\Phi$ is injective.
\end{proof}


\section{Relative branching systems and properties of $C^*(\G,X)$.}\label{BSU}
Our main goal in this section is to relate the theory of branching systems for ultragraphs—which
yields representations of $C^*(\mathcal{G})$ (see \cite{DDH})—to a parallel theory, which we call
\emph{relative branching systems}, that yields representations of $C^*(\mathcal{G},X)$. This
connection is mediated by the isomorphism of
Theorem~\ref{teo iso algebras ultragrafos e relativas}; in particular, the uniqueness theorems
proved earlier admit corresponding formulations in the setting of representations arising from relative branching systems.

\begin{definition}\label{relativebs}
Let $(\mathcal{G},X)$ be a relative ultragraph, let $(\mathcal{X},\mu)$ be a measure space, and
let $\{R_e,D_A\}_{e\in \mathcal{G}^1,\;A\in \mathcal{E}}$ be a family of measurable subsets of
$\mathcal{X}$. Suppose that
\begin{enumerate}
\item\label{R_e cap R_f =emptyset if e neq f}
$R_e \cap R_f \overset{\mu\text{-a.e.}}{=} \emptyset$ if $e \neq f \in \mathcal{G}^1$.
\item
$D_\emptyset=\emptyset$, \quad
$D_A \cap D_B \overset{\mu\text{-a.e.}}{=} D_{A \cap B}$, \quad
$D_A \cup D_B \overset{\mu\text{-a.e.}}{=} D_{A \cup B}$ \ for all $A,B \in \mathcal{E}$.
\item
$R_e \overset{\mu\text{-a.e.}}{\subseteq} D_{s(e)}$ for all $e\in \mathcal{G}^1$.
\item\label{D_v=cup_{e in s^{-1}(v)}R_e}
$D_v \overset{\mu\text{-a.e.}}{=} \displaystyle\bigcup_{e \in s^{-1}(v)} R_e$ for all $v \in X$.
\item
For each $e\in \mathcal{G}^1$, there exist measurable maps
$f_e: D_{r(e)} \to R_e$ and $f_e^{-1}: R_e \to D_{r(e)}$ such that
$f_e \circ f_e^{-1} \overset{\mu\text{-a.e.}}{=} \mathrm{id}_{R_e}$ and
$f_e^{-1} \circ f_e \overset{\mu\text{-a.e.}}{=} \mathrm{id}_{D_{r(e)}}$; moreover
$\mu \circ f_e \ll \mu$ on $D_{r(e)}$ and $\mu \circ f_e^{-1} \ll \mu$ on $R_e$.
\end{enumerate}
Denote the Radon–Nikodym derivatives by
\[
\Phi_{f_e} \;:=\; \frac{d(\mu \circ f_e)}{d\mu}\ \text{ on } D_{r(e)},
\qquad
\Phi_{f_e^{-1}} \;:=\; \frac{d(\mu \circ f_e^{-1})}{d\mu}\ \text{ on } R_e.
\]
We call $\{R_e,D_A,f_e\}_{e \in \mathcal{G}^1,\;A \in \mathcal{E}}$ a
\emph{relative $(\mathcal{G},X)$-branching system} on $(\mathcal{X},\mu)$.
\end{definition}

\begin{remark}
If $X=\mathrm{Reg}(\mathcal{G})$ in Definition~\ref{relativebs}, we obtain a
$\mathcal{G}$-branching system as in \cite{DDH}. The only difference from the original definition
is that condition~\eqref{D_v=cup_{e in s^{-1}(v)}R_e} is required to hold only for the selected
set of regular vertices $X$.
\end{remark}


Let $\mathcal{G}$ be an ultragraph, and let
$\{R_e, D_A, f_e\}_{e \in \mathcal{G}^1,\, A \in \mathcal{E}}$ be a
$\mathcal{G}$-branching system on $(\mathcal{X}, \mu)$.
Since the domains of $\Phi_{f_e^{-1}}$ and $f_e^{-1}$ both equal $R_e$, we may regard them as
measurable maps on $\mathcal{X}$ by extending them to be $0$ outside $R_e$. Thus, for each
$\phi \in \mathcal{L}^2(\mathcal{X}, \mu)$, we can consider the function
\[
\Phi_{f_e^{-1}}^{1/2}\, (\phi \circ f_e^{-1})\, \chi_{R_e}.
\]
Similarly, extending $f_e$ and $\Phi_{f_e}$ by $0$ outside $D_{r(e)}$, we obtain the function
\[
\Phi_{f_e}^{1/2}\, (\phi \circ f_e)\, \chi_{D_{r(e)}}.
\]
The following theorem, proved in \cite{DDH}, shows how a branching system yields a representation.

\begin{theorem}[\cite{DDH}, Theorem~4.1]\label{repinducedbybranchingsystems}
Let $\mathcal{G}$ be an ultragraph and let
$\{R_e,D_A,f_e\}_{e\in \mathcal{G}^1,\, A\in \mathcal{E}}$ be a
$\mathcal{G}$-branching system on a measure space $(\mathcal{X},\mu)$. Then there exists a unique
representation $\pi:C^*(\mathcal{G}) \to B(\mathcal{L}^2(\mathcal{X},\mu))$ such that, for all
$e \in \mathcal{G}^1$, $A \in \mathcal{E}$, and $\phi \in \mathcal{L}^2(\mathcal{X},\mu)$,
\[
\pi(s_e)(\phi)=\Phi_{f_e^{-1}}^{1/2}\,(\phi \circ f_e^{-1})\,\chi_{R_e},
\qquad
\pi(p_A)(\phi)=\chi_{D_A}\,\phi.
\]
\end{theorem}

An identical proof shows that a relative $(\mathcal{G},X)$-branching system induces a
representation of the relative ultragraph algebra $C^*(\mathcal{G},X)$ defined in the previous
section. For future reference we state the result:

\begin{proposition}\label{relativerepinducedbybranchingsystems}
Let $(\mathcal{G}, X)$ be a relative ultragraph, and let
$\{R_e,D_A,f_e\}_{e\in \mathcal{G}^1,\, A\in \mathcal{E}}$ be a relative $(\mathcal{G},X)$-branching
system on a measure space $(\mathcal{X},\mu)$. Then there exists a unique representation
$\pi:C^*(\mathcal{G},X) \to B(\mathcal{L}^2(\mathcal{X},\mu))$ such that, for all
$e \in \mathcal{G}^1$, $A \in \mathcal{E}$, and $\phi \in \mathcal{L}^2(\mathcal{X},\mu)$,
\[
\pi(s_e)(\phi)=\Phi_{f_e^{-1}}^{1/2}\,(\phi \circ f_e^{-1})\,\chi_{R_e},
\qquad
\pi(p_A)(\phi)=\chi_{D_A}\,\phi.
\]
\end{proposition}

\begin{remark}
By standard computations (see \cite{DDH} for details) one obtains, for $\phi \in \mathcal{L}^2(\mathcal{X},\mu)$,
\[
\pi(s_e)^*(\phi) = \Phi_{f_e}^{1/2}\, (\phi \circ f_e)\, \chi_{D_{r(e)}},
\qquad
\pi(s_e s_e^*)(\phi) = \chi_{R_e}\,\phi.
\]
We will use these identities without explicit reference. We also note that the usual set-theoretic
and (bijective) function properties hold for the sets $D_A$, $R_e$ and the maps $f_e$, $f_e^{-1}$
up to sets of measure zero; i.e., they are valid almost everywhere.
\end{remark}

Let $\mathcal{G}=(\mathcal{G}^0,\mathcal{G}^1,r,s)$ be an ultragraph and let $X\subseteq \mathrm{Reg}(\mathcal{G})$.
In what follows we show how to produce a branching system for the ultragraph algebra
$C^*(\mathcal{G}_X)$ starting from a $(\mathcal{G},X)$-relative branching system. Throughout this
section we keep the notation $Y:=\mathrm{Reg}(\mathcal{G})\setminus X$.

Let $\{R_e,D_A,f_e\}_{e \in \mathcal{G}^1,\,A \in \mathcal{E}}$ be a relative
$(\mathcal{G},X)$-branching system on $(\mathcal{X},\mu)$. Recall from
Lemma~\ref{decomposition1} that if $Z\in \mathcal{E}_X$ then $Z=A \cup (B\cap Y)'$ with
$A,B \in \mathcal{E}$. Using this, for each $Z\in\mathcal{E}_X$ define
\[
B_Z \;:=\; \bigl(D_A \setminus D_{A\cap Y}\bigr)
\;\cup\;
\bigcup_{\substack{g \in \mathcal{G}^1 \\ s(g)\in A\cap Y}} R_g
\;\cup\;
\bigcup_{v \in B\cap Y}
\left(
D_v \setminus \bigcup_{\substack{g \in \mathcal{G}^1 \\ s(g)=v}} R_g
\right).
\]

Recall also that $\mathcal{G}_X^1=\mathcal{G}^1 \cup C$, where
\[
C \;=\; \{\, e' \mid e\in \mathcal{G}^1 \text{ and } r(e)\cap Y \neq \emptyset \,\}.
\]
Define, for $e\in \mathcal{G}^1$,
\[
Q_e
\;:=\;
\Bigl(f_e\bigl(D_{r(e)}\bigr)\setminus f_e\bigl(D_{r(e)\cap Y}\bigr)\Bigr)
\;\cup\;
f_e\!\left(\,\bigcup_{\substack{g \in \mathcal{G}^1 \\ s(g)\in r(e)\cap Y}} R_g \right),
\]
and, for $e'\in C$,
\[
Q_{e'}
\;:=\;
f_e\!\left(
\;\bigcup_{v \in r(e)\cap Y}
\left(
D_v \setminus \bigcup_{\substack{g \in \mathcal{G}^1 \\ s(g)=v}} R_g
\right)
\right).
\]

Note that $B_{r(e)},\,B_{r(e')}\subseteq D_{r(e)}$. Thus we can define the maps
\[
g_e:\, B_{r(e)} \longrightarrow g_e\bigl(B_{r(e)}\bigr),
\qquad
g_{e'}:\, B_{r(e')} \longrightarrow g_{e'}\bigl(B_{r(e')}\bigr),
\]
as the restrictions of $f_e$ to $B_{r(e)}$ and $B_{r(e')}$, respectively. By construction,
\[
g_e\bigl(B_{r(e)}\bigr)=f_e\bigl(B_{r(e)}\bigr)=Q_e,
\qquad
g_{e'}\bigl(B_{r(e')}\bigr)=f_e\bigl(B_{r(e')}\bigr)=Q_{e'},
\]
and consequently the inverse maps $g_e^{-1}$ and $g_{e'}^{-1}$ are also well defined.

To show that the construction above produces a branching system, we first prove the following
lemma.

\begin{lemma}\label{intersecao vazia}
Let $(\mathcal{G},X)$ be a relative ultragraph, and set $Y=\mathrm{Reg}(\mathcal{G})\setminus X$,
with $r(e)\cap Y$ finite for all $e\in \mathcal{G}^1$. If
$Z_1=A_1 \cup (B_1 \cap Y)$ and $Z_2=A_2 \cup (B_2 \cap Y)$ are the decompositions given by
Lemma~\ref{decomposition1}, then:
\begin{itemize}
    \item $( D_{A_1} \setminus D_{A_1 \cap Y} ) \cap \left( \bigcup\limits_{g \ | \ s(g) \in A_2 \cap Y} R_g \right)  \stackrel{\mu-a.e.}{=} \emptyset.$
    \item $ \left( \bigcup\limits_{g \ | \ s(g) \in A_1 \cap Y} R_g \right) \cap ( D_{A_2} \setminus D_{A_2 \cap Y} ) \stackrel{\mu-a.e.}{=} \emptyset$.
    \item $( D_{A_1} \setminus D_{A_1 \cap Y} ) \cap \left(  \bigcup\limits_{v \in B_2 \cap Y} \left(D_v \setminus \bigcup\limits_{ e \ | \ s(e) = v} R_e \right) \right) \stackrel{\mu-a.e.}{=} \emptyset$.
    \item $\left(  \bigcup\limits_{v \in B_1 \cap Y} \left(D_v \setminus \bigcup\limits_{ e \ | \ s(e) = v} R_e \right) \right) \cap ( D_{A_2} \setminus D_{A_2 \cap Y} ) \stackrel{\mu-a.e.}{=} \emptyset$.
    \item $\left( \bigcup\limits_{g \ | \ s(g) \in A_1 \cap Y} R_g \right) \cap \left(  \bigcup\limits_{v \in B_2 \cap Y} \left(D_v \setminus \bigcup\limits_{ e \ | \ s(e) = v} R_e \right) \right) \stackrel{\mu-a.e.}{=} \emptyset$.
     \item $ \left(  \bigcup\limits_{v \in B_1 \cap Y} \left(D_v \setminus \bigcup\limits_{ e \ | \ s(e) = v} R_e \right) \right) \cap \left( \bigcup\limits_{g \ | \ s(g) \in A_2 \cap Y} R_g \right) \stackrel{\mu-a.e.}{=} \emptyset$.
\end{itemize}  
\end{lemma}

\begin{proof}
    It follows from the properties of the sets $R_g$ ($g \in \G^1$) and $D_{A}$ ($A \in \mathcal{E}$).
    
\end{proof}

\begin{proposition}
Let $(\mathcal{G},X)$ be a relative ultragraph and set $Y=\mathrm{Reg}(\mathcal{G})\setminus X$,
assuming that $r(e)\cap Y$ is finite for all $e\in \mathcal{G}^1$.  
Then the collection $\{Q_f, B_Z, g_f\}_{f\in \mathcal{G}_X^1,\; Z\in \mathcal{E}_X}$ constructed above 
forms a $\mathcal{G}_X$-branching system on $(\mathcal{X},\mu)$.
\end{proposition}
  
\begin{proof}
\begin{enumerate} 
 We will check the conditions of Definition \ref{relativebs}.
  
\item\label{Q_e cap R_Q =emptyset if e neq f} $Q_e\cap Q_f \stackrel{\mu-a.e.}{=} \emptyset$ if $e \neq f \in \mathcal{G}^1_X$.

First, note that $Q_e = f_e(B_{r(e)}) \stackrel{\mu-a.e}{\subseteq} R_e$ for all $e \in \G^1$ and $Q_{e'} = f_e(B_{r(e')})  \stackrel{\mu-a.e}{\subseteq} R_e$  whenever $e'$ is defined. As $R_e \cap R_f \stackrel{\mu-a.e}{=} \emptyset$ the claim is immediate in all cases except for the case involving $e$ and $e'$. In this case we observe that, by Lemma~\ref{intersecao vazia},
$$\left( (D_{r(e)} \setminus D_{r(e) \cap Y}) \cup \left( \bigcup\limits_{g \ | \ s(g) \in r(e) \cap Y} R_g \right) \right) \cap \left( \bigcup\limits_{v \in r(e) \cap Y} (D_v \setminus \bigcup\limits_{g \ | \ s(g) = v} R_g ) \right) \stackrel{\mu-a.e}{=} \emptyset.$$ 

Recall from the definition of $Q_e$ and $Q_{e'}$ that $Q_e$ is the image by $f_e$ of the left side of the above intersection, while $Q_{e'}$ is the image of $f_e$ of the right side. Using that $\mu \circ f_e \ll \mu$ we conclude the desired claim.

\item   It is not hard to see that $B_\emptyset=\emptyset$ and $B_{Z_1} \cup B_{Z_2} \stackrel{\mu-a.e.}{=} B_{Z_1 \cup Z_2}$  for all $Z_1, Z_2 \in \mathcal{E}_X$. The equality $ B_{Z_1} \cap B_{Z_2}\stackrel{\mu-a.e.}{=} B_{Z_1 \cap Z_2}$ follows from the previous lemma.

\item We leave the equality $Q_e\stackrel{\mu-a.e.}{\subseteq}B_{s(e)}$ for all $e\in \mathcal{G}^1_X$ to the reader.

\item  To see why $B_v\stackrel{\mu-a.e.}{=} \bigcup\limits_{e \in \G_X^1 \ | \ s_X(e) = v}Q_e$ for all $v \in Reg(\G_X)$ we will prove first that $Q_e \cup Q_{e'} \stackrel{\mu-a.e}{=} f_e({D_{r(e)}})$ whenever $e'$ is defined, i.e, $r(e) \cap Y \neq \emptyset$ (Note that if $r(e) \cap Y = \emptyset$ then obviously $Q_e = f_e(D_{r(e)})$).

As $r(e) \cap Y = \{v_1, \ldots, v_N\}$ for some $N \in \mathbb{N}$ then $D_{r(e) \cap Y} \stackrel{\mu-a.e}{=} D_{v_1} \cup \ldots \cup D_{v_N}$ and consequently:

$$D_{r(e) \cap Y} \stackrel{\mu-a.e}{=}
\left( \bigcup\limits_{g \ | \ s(g) \in r(e) \cap Y} R_g \right) \cup \left( \bigcup\limits_{v \in r(e) \cap Y} \left(  D_v \setminus \bigcup\limits_{e \ | \ s(e) = v} R_e \right) \right)$$
because if $x \in D_{r(e) \cap Y}$ then $x \in D_{v_i}$ for some $i = 1, \ldots, N$ except for a set of null measure. Therefore
$$f_e(D_{r(e)}) = f_e(D_{r(e)} \setminus D_{r(e) \cap Y}) \cup f_e(D_{r(e) \cap Y}) = Q_e \cup Q_{e'}.$$

Now, using the previous work, we can deduce the following crucial equality:
\begin{align*}
    \bigcup\limits_{e \in \G^1 \ | \ s(e) = v}R_e & \stackrel{\mu-a.e.}{=}  \left( \bigcup\limits_{e \in \G^1, r(e) \cap Y \neq \emptyset \ | \ s(e) = v}R_e  \right) \cup \left( \bigcup\limits_{e \in \G^1, r(e) \cap Y = \emptyset \ | \ s(e) = v}R_e  \right)\\
    & \stackrel{\mu-a.e.}{=}  \left( \bigcup\limits_{e \in \G^1, r(e) \cap Y \neq \emptyset \ | \ s(e) = v} Q_e \cup Q_{e'} \right) \cup \left( \bigcup\limits_{e \in \G^1, r(e) \cap Y = \emptyset \ | \ s(e) = v}Q_e \right)\\
     & \stackrel{\mu-a.e.}{=}  \left( \bigcup\limits_{e \in \G^1_X  \ | \ s_X(e) = v} Q_e \right). 
\end{align*}
To finish this item is enough to note that for all $v \in Reg(\G_X) = Reg(\G)$ we have
$$B_v  \stackrel{\mu-a.e.}{=} \begin{cases}
    D_v \textrm{ if } v\notin Y. \\
    \bigcup\limits_{e \in \G^1 \ | \ s(e) = v} R_e \textrm{ if } v \in Y.\\
\end{cases}$$
By item $(4)$ from the definition of relative branching system (since $v \in X$) we conclude that $$B_v \stackrel{\mu-a.e}{=}  \bigcup\limits_{e  \in \G^1 \ | \ s(e) = v} R_e$$
for all $v \in Reg(\G)$. Therefore, by the computation of the last paragraph we are done.

\item For each $e \in \mathcal{G}^1_X$ we note that the functions $g_e$ was defined as $\mu-a.e$ bijections . Moreover, each $g_e$ is a restriction of $f_e$ and therefore item $(5)$ of the definition of branching system follows.
\end{enumerate}
\end{proof}

We now relate the representations arising from a relative branching system and from its
lift to $\mathcal{G}_X$.
\begin{theorem}\label{DiagramadeBS}
Let $(\mathcal{G}, X)$ be a relative ultragraph and set $Y=\mathrm{Reg}(\mathcal{G})\setminus X$,
with $r(e)\cap Y$ finite for all $e \in \mathcal{G}^1$. Let
$\{R_e,D_A,f_e\}_{e\in \mathcal{G}^1,\, A\in \mathcal{E}}$ be a $(\mathcal{G}, X)$-relative
branching system on a measure space $(\mathcal{X},\mu)$, and consider the
$\mathcal{G}_X$-branching system $\{Q_f, B_Z, g_f\}_{f \in \mathcal{G}_X^1,\, Z \in \mathcal{E}_X}$
on $(\mathcal{X},\mu)$ constructed above. Let $\eta$ be the representation induced by
$\{R_e,D_A,f_e\}$ via Proposition ~\ref{relativerepinducedbybranchingsystems}, and let $\pi$ be the
representation induced by $\{Q_f, B_Z, g_f\}$ via Theorem ~\ref{repinducedbybranchingsystems}.
Then $\pi=\eta \circ \psi$, where $\psi$ is the isomorphism given in
Theorem~\ref{teo iso algebras ultragrafos e relativas}. 
\end{theorem}
\begin{proof}
    Let $Z \in \mathcal{E}_X$ and write $Z = A \cup (B \cap Y)'$ via Lemma \ref{decomposition1}. Note that for all $F \in L^2(\mathcal{X}, \mu)$ we have
    \begin{align*}
        & (\eta \circ \psi)(P_Z)(F)  = \eta \left( p_A - p_{A \cap Y} + \sum\limits_{g \ | \ s(g) \in A \cap Y}s_g s_g^* + \sum\limits_{v \in B \cap Y}q_v \right)(F) \\
        & = \chi_{D_A} \cdot F - \chi_{D_{A \cap Y}} \cdot F + \sum\limits_{g \ | \ s(g) \in A \cap Y} \chi_{R_g} \cdot F +  \sum\limits_{v \in B \cap Y} \left(  \chi_{D_v} \cdot F - \sum\limits_{g \ | \ s(g) = v } \chi_{R_g} \cdot F \right)     \\      
        & = \left(  \chi_{D_{A}} - \chi_{D_{A \cap Y}} + \chi_{  \left(  \bigcup\limits_{g \ | \ s(g) \in A \cap Y}R_g  \right) }  +  \chi_{  \left(  \bigcup\limits_{\ v \in B \cap Y } ( D_v \setminus \cup_{ \{g \ | \ s(g) = v \} } R_g )  \right) } \right) \cdot F = \pi(P_Z)(F) \\
    \end{align*}
    where we used that $\eta(s_g s_g^*)(F) = \chi_{R_g} \cdot F$. Also the fact that $R_g \cap R_h  \stackrel{\mu-a.e.}{=}\emptyset$ if $g \neq h$ and $D_v \cap D_w  \stackrel{\mu-a.e.}{=}\emptyset$ if $v \neq w$ is crucial to write the sums of characteristic functions as the characteristic function of a union. The last equality follows from the definition of $B_Z$, because $\pi(P_Z)(F) = \chi_{B_Z} \cdot F$. 

    We now check that $\pi(S_f) = (\eta \circ \psi)(S_f)$ for $f \in \mathcal{G}_X^1$.. First, let  $f = e \in \mathcal{G}^1$ and fix an $F \in L^2(\mathcal{X}, \mu)$. Then

   \begin{align*}
   & (\eta \circ \psi)(S_e)(F)  = \eta\left(s_e - s_ep_{r(e) \cap Y} + \sum\limits_{g \ | \ s(g) \in r(e) \cap Y}s_e s_g s_g^* \right)(F) \\
   & = \Phi_{f_e^{-1}}^{\frac{1}{2}} \cdot (F \circ f_e^{-1} ) - \Phi_{f_e^{-1}}^{\frac{1}{2}} \cdot \left( \chi_{f_e(D_{r(e) \cap Y})} \cdot F \circ f_e^{-1} \right)  + \Phi_{f_e^{-1}}^{\frac{1}{2}} \cdot \left( \chi_{ 
  f_e \left( \bigcup\limits_{ g \ | \ s(g) \in r(e) \cap Y} R_g  \right)  }      \cdot  F \circ f_e^{-1} \right) \\
  & = \Phi_{f_e^{-1}}^{\frac{1}{2}} \cdot (F \circ f_e^{-1} )  \cdot \left(  \chi_{R_e} - \chi_{f_e(D_{r(e) \cap Y})} +   \chi_{ 
  f_e \left( \bigcup\limits_{ g \ | \ s(g) \in r(e) \cap Y} R_g  \right)  }   \right) \\
  & = \Phi_{f_e^{-1}}^{\frac{1}{2}} \cdot (F \circ f_e^{-1} )  \cdot \left(  \chi_{f_e(D_{r(e)})} - \chi_{f_e(D_{r(e) \cap Y})} +   \chi_{ 
  f_e \left( \bigcup\limits_{ g \ | \ s(g) \in r(e) \cap Y} R_g  \right)  }   \right) \\
  & = \Phi_{f_e^{-1}}^{\frac{1}{2}} \cdot (F \circ f_e^{-1} )  \cdot  \chi_{Q_e} = \pi(S_e)(F).\\
\end{align*}

Hence, for $f=e \in \mathcal{G}^1$ we have $\pi(S_f)=(\eta \circ \psi)(S_f)$.

Therefore, $\pi=\eta \circ \psi$ on the generators of $C^*(\mathcal{G}_X)$, which proves the theorem.
\end{proof}

In the previous section we derived results for the relative ultragraph algebra by using the
isomorphism between $C^*(\mathcal{G},X)$ and $C^*(\mathcal{G}_X)$ together with known results
for ultragraph $C^*$-algebras. By the preceding theorem, statements involving relative branching
systems can likewise be deduced from the corresponding results for branching systems. We present a few of these results below.

\begin{proposition}\label{injectivitybscondL}
Let $(\mathcal{G}, X)$ be an ultragraph and set $Y=\mathrm{Reg}(\mathcal{G})\setminus X$.
Suppose that $r(e) \cap Y$ is finite for all $e \in \mathcal{G}^1$. Let
$\{R_e,D_A,f_e\}_{e\in \mathcal{G}^1,\; A\in \mathcal{E}}$ be a $(\mathcal{G}, X)$-relative
branching system on a measure space $(\mathcal{X},\mu)$, and let
\[
\eta: C^*(\mathcal{G},X) \longrightarrow B\!\bigl(L^2(\mathcal{X},\mu)\bigr)
\]
be the representation induced via the branching system. Suppose that:
\begin{enumerate}
    \item $\mu(D_v) > 0$ for all $v \notin Y$;
    \item $\mu(R_e) > 0$ for all $e \in \mathcal{G}^1$ with $s(e) \in Y$;
    \item $\mu\!\left( D_v \setminus \bigcup_{\substack{e \in \mathcal{G}^1\\ s(e)=v}} R_e \right) > 0$
    for all $v \in Y$.
\end{enumerate}
If $\mathcal{G}$ satisfies the Relative Condition~(L), then $\eta$ is injective.
\end{proposition}

\begin{proof}
Recall that the representation $\eta$ induced by the relative branching system satisfies
$\eta(p_A)(F) = \chi_{D_A}\,F$ and $\eta(s_e s_e^*)(F) = \chi_{R_e}\,F$
for all $F \in L^2(\mathcal{X},\mu)$. By (1) and (2) we see that
$\eta(p_A) \neq 0$ for all nonempty $A \in \mathcal{E}$ with $A \cap Y = \emptyset$, and
$\eta(s_e s_e^*) \neq 0$ for all $e \in \mathcal{G}^1$ with $s(e) \in Y$.
Moreover, for each $v \in Y$,
\begin{align*}
\eta(q_v)(F)
&= \eta\!\left(p_v - \sum_{\substack{e \in \mathcal{G}^1\\ s(e)=v}} s_e s_e^*\right)(F)
  = \chi_{D_v}\,F - \sum_{\substack{e \in \mathcal{G}^1\\ s(e)=v}} \chi_{R_e}\,F \\
&= \biggl(\chi_{D_v} - \chi_{\displaystyle\bigcup_{\substack{e \in \mathcal{G}^1\\ s(e)=v}} R_e}\biggr) F
  = \chi_{\,D_v \setminus \displaystyle\bigcup_{\substack{e \in \mathcal{G}^1\\ s(e)=v}} R_e}\,F,
\end{align*}
and thus, by (3), $\eta(q_v) \neq 0$ for all $v \in Y$. The result now follows from
Theorem~\ref{Relative condition (L) implies injective}.
\end{proof}





Without assuming the Relative Condition~(L), we obtain the following result.

\begin{proposition}\label{injectivitybs}
Let $(\mathcal{G}, X)$ be a relative ultragraph and set $Y=\mathrm{Reg}(\mathcal{G})\setminus X$,
with $r(e)\cap Y$ finite for all $e \in \mathcal{G}^1$. Let
$\{R_e,D_A,f_e\}_{e\in \mathcal{G}^1,\, A\in \mathcal{E}}$ be a $(\mathcal{G}, X)$-relative
branching system on a measure space $(\mathcal{X},\mu)$, and let $\eta$ be the representation
induced by $\{R_e,D_A,f_e\}$ via Proposition~\ref{relativerepinducedbybranchingsystems}. Suppose:
\begin{enumerate}
    \item $\mu(D_v) > 0$ for all $v \notin Y$;
    \item $\mu(R_e) > 0$ for all $e \in \mathcal{G}^1$ with $s(e) \in Y$;
    \item $\mu\!\left( D_v \setminus \bigcup_{\substack{e \in \mathcal{G}^1\\ s(e)=v}} R_e \right) > 0$
          for all $v \in Y$;
    \item for any simple cycle $\alpha=(\alpha_i)_{i=1}^N$ without exits in $\mathcal{G}$ such that
          $r(\alpha_i) \notin Y$ for all $i=1,\ldots,N$, and for any finite subset
          $\mathcal{F}\subseteq\mathbb{N}$, there exists a measurable set
          $E \subseteq D_{s(\alpha)}$ with $\mu(E)\neq 0$ such that
          $\mu\!\bigl(f_{n\alpha}(E)\cap E\bigr)=0$ for all $n\in \mathcal{F}$.
\end{enumerate}
If {\rm(1)–(4)} hold, then $\eta$ is injective. Moreover, if $\eta$ is injective, then
{\rm(1)–(3)} hold.
\end{proposition}

\begin{proof}

We will start by showing that if $(1)-(4)$ hold then $\eta$ is injective. By Theorem \ref{DiagramadeBS} we know that $\eta \circ \psi = \pi$, where $\pi$ is the representation induced by the branching system $\{Q_f, B_Z, g_f\}_{f \in \G_X^1, Z \in \mathcal{E}_X}$ in $(\mathcal{X},\mu)$. Since $\psi$ is an isomorphism, to show that $\eta$ is injective is enough to prove that $\pi$ is injective. To prove that $\pi$ is injective, we will show that $\pi$ satisfies the the hypothesis of Theorem $8.3$ in \cite{DDH}.
If $Z \in \mathcal{E}_X$ is non-empty then we can write $Z = A \cup (B \cap Y)'$, where $A,B \in \mathcal{E}$ due to Lemma~\ref{decomposition1}. Recall that:
$$B_{Z} := ( D_{A} - D_{A \cap Y} )\cup \left( \bigcup\limits_{g \ | \ s(g) \in A \cap Y } R_g \right)  \cup  \left( \bigcup\limits_{v \in B \cap Y} \left( D_v -    \bigcup\limits_{g \ | \ s(g) = v} R_g  \right) \right).$$

If $B \cap Y \neq \emptyset$, choose $v_0 \in B \cap Y$. Then
\[
D_{v_0} \setminus \bigcup_{\substack{e \in \mathcal{G}^1 \\ s(e)=v_0}} R_e
\ \subseteq\
\bigcup_{v \in B \cap Y} \left( D_v \setminus \bigcup_{\substack{g \in \mathcal{G}^1 \\ s(g)=v}} R_g \right)
\ \subseteq\ B_Z,
\]
and by (3) it follows that $\mu(B_Z) > 0$. If $B \cap Y = \emptyset$, then $A$ must be nonempty (otherwise
$Z$ would be empty). If $A \cap Y \neq \emptyset$, pick $v_0 \in A \cap Y$. Since $Y$ has no sinks, there exists $e \in \mathcal{G}^1$ with $s(e)=v_0$, and hence
\[
R_e \ \subseteq\ \bigcup_{\substack{g \in \mathcal{G}^1 \\ s(g)\in A \cap Y}} R_g \ \subseteq\ B_Z.
\]
Therefore, by (2) we obtain $\mu(B_Z) > 0$. Finally, if $A \cap Y = \emptyset$, then
\[
D_A \ =\ D_A \setminus D_{A \cap Y} \ \subseteq\ B_Z,
\]
and thus again $\mu(B_Z) > 0$. We conclude that $\mu(B_Z) > 0$ for every nonempty $Z \in \mathcal{E}_X$.

Now let $\alpha=(\alpha_i)_{i=1}^N$ be a simple cycle without exits in $\mathcal{G}_X$, and let
$\mathcal{F} \subset \mathbb{N}$ be finite. By Lemma~\ref{lemma relative condition L}, $\alpha$ is a cycle without exits in $\mathcal{G}$ with $r(\alpha_i) \notin Y$ for all $i=1,\dots,N$. By (4) there exists a measurable set $E \subseteq D_{s(\alpha)}$ with $\mu(E) \neq 0$ such that
$\mu\bigl(f_{n\alpha}(E) \cap E\bigr)=0$ for all $n \in \mathcal{F}$. Since
$s(\alpha)=r(\alpha_N)\notin Y$, we have $B_{s(\alpha)}=D_{s(\alpha)}$ by definition; moreover, by the construction of the branching system $\{Q_f,B_Z,g_f\}_{f \in \mathcal{G}_X^1,\; Z \in \mathcal{E}_X}$ we have $f_{n\alpha}=g_{n\alpha}$ (because $B_{r(\alpha_i)}=D_{r(\alpha_i)}$ for all $i$ with $r(\alpha_i)\notin Y$). Thus $\pi$ satisfies the hypotheses of \cite[Theorem~8.3]{DDH}, as required. We conclude that $\pi$ is injective.

Now suppose that $\eta$ is injective. Since $\eta$ is induced by the relative branching system, we
have, for all $F \in \mathcal{L}^2(\mathcal{X},\mu)$,
\[
\eta(p_v)(F) = \chi_{D_v}\,F
\qquad\text{and}\qquad
\eta(s_e s_e^*)(F) = \chi_{R_e}\,F .
\]
Therefore, (1) and (2) must hold; otherwise $\eta$ would annihilate a nonzero projection, which
contradicts injectivity.

To prove (3), use Theorem~\ref{DiagramadeBS} to write $\eta \circ \psi = \pi$, so $\pi$ is
injective. In particular, $\pi(P_v) \neq 0$ for every $v \in Y$, and hence
\[
0 \neq \pi(P_v) = \eta \circ \psi(P_v)
= \eta\!\left( p_v - \sum_{\substack{e \in \mathcal{G}^1\\ s(e)=v}} s_e s_e^* \right).
\]
Repeating the computation from the previous theorem, for $F \in \mathcal{L}^2(\mathcal{X},\mu)$ we obtain
\[
\eta\!\left( p_v - \sum_{\substack{e \in \mathcal{G}^1\\ s(e)=v}} s_e s_e^* \right)(F)
= \chi_{\,D_v \setminus \displaystyle\bigcup_{\substack{e \in \mathcal{G}^1\\ s(e)=v}} R_e}\,F,
\]
and thus (3) follows.

\end{proof}

The next (and final) result of this section shows that the converse of the previous theorem holds
when the counting measure is used.

\begin{theorem}\label{injectivitybscountingmeasure}
Let $(\mathcal{G}, X)$ be a relative ultragraph and set $Y=\mathrm{Reg}(\mathcal{G})\setminus X$,
with $r(e)\cap Y$ finite for all $e \in \mathcal{G}^1$. Let
$\{R_e,D_A,f_e\}_{e\in \mathcal{G}^1,\, A\in \mathcal{E}}$ be a $(\mathcal{G}, X)$-relative
branching system on a measure space $(\mathcal{X},\mu)$, where $\mu$ is the counting measure, and
let $\eta$ be the representation induced by $\{R_e,D_A,f_e\}$ via
Proposition~\ref{relativerepinducedbybranchingsystems}. Then $\eta$ is injective if and only if:
\begin{enumerate}[a)]
    \item $D_v \neq \emptyset$ for all $v \notin Y$;
    \item $R_e \neq \emptyset$ for all $e \in \mathcal{G}^1$ with $s(e) \in Y$;
    \item $D_v \setminus \displaystyle\bigcup_{\substack{e \in \mathcal{G}^1\\ s(e)=v}} R_e \neq \emptyset$
          for all $v \in Y$;
    \item for any simple cycle $\alpha=(\alpha_i)_{i=1}^N$ without exits such that
          $r(\alpha_i)\notin Y$ for all $i=1,\ldots,N$, and for any finite subset
          $\mathcal{F}\subseteq\mathbb{N}$, there exists $x \in D_{s(\alpha)}$ with
          $f_{n\alpha}(x) \neq x$ for all $n \in \mathcal{F}$.
\end{enumerate}
\end{theorem}

\begin{proof}
 If conditions $a) - d)$ above holds, then choosing $E = \{x\}$ in condition $(4)$ of Proposition \ref{injectivitybs} we see that $\eta$ is injective.

If $\eta$ is injective, then $(a)-(c)$ holds again by Proposition \ref{injectivitybs}. So, is enough to prove that $(d)$ holds. Let $\alpha = (\alpha_i)_{i = 1}^N$ be a simple cycle without exits such that $r(\alpha_i) \notin Y$ for all $i = 1, \ldots, N$ and for any finite subset $\mathcal{F} \subseteq \mathbb{N}$. 

If there exists $x \in D_{s(\alpha)}$ such that $f_{n\alpha}(x) \neq x$ for all $n \in \mathcal{F}$ we are done. If such an element $x \in D_{s(\alpha)}$ does not exists we will produce a contradiction. In fact, if that is the case then for all $x \in D_{s(\alpha)}$ there exists $n_x \in \mathcal{F}$ such that $f_{n_x \alpha}(x) = x$. As $\mathcal{F}$ is finite we can consider $t:= {\displaystyle \prod_{n \ | \ n \in \mathcal{F}} n }$; so the for all $x \in D_{s(\alpha)}$, the number $t$ is a multiple of $n_x$ and consequently 

$$f_{t\alpha}(x) = \underbrace{f_{\alpha} \circ f_{\alpha} \circ \ldots \circ f_{\alpha}}_{t \,\ \textrm{times } } (x) = \underbrace{f_{\alpha} \circ \ldots \circ f_{\alpha}}_{n_x \,\ \textrm{ times } } \circ \ldots \circ \underbrace{ f_{\alpha} \ldots \circ f_{\alpha}}_{n_x \,\ \textrm{times } } (x) = x $$
for all $x \in D_{s(\alpha)}$ showing that $f_{t\alpha} = Id$ in $D_{s(\alpha})$. Now, as $\mu$ is the counting measure, the Radon-Nikodym derivatives that appears in the branching system are the constant functions $1$ and thus the representation $\eta$ satisfies:

$$\eta(s_{t\alpha})(F) = \chi_{D_{s(\alpha)}} \cdot F \circ f_{t \alpha}^{-1} = \chi_{D_{s(\alpha)}} \cdot F \,\ \,\ \,\ \forall F \in L^2(\mu);$$
$$\eta(p_{s(\alpha})(F) = \chi_{D_{s(\alpha)}} \cdot F \,\ \,\ \,\ \forall F \in L^2(\mu).$$
Therefore $\eta(s_{t\alpha}) = \eta(p_{s(\alpha)})$ which contradicts the injectivity of $\eta$.
\end{proof}



\section{Markov interval maps with escape sets and associated representations}
\label{secmarkovescape2}

In this section, we introduce infinite partition Markov interval maps with escape sets and show how they induce representations of relative ultragraph algebras.

For a set $I$ we denote by $\mathring{I}$ its interior.

\begin{definition}\label{Markovclass}
    Let $I = [a,b) \subseteq \mathbb{R}$ be an interval, where $a \in \mathbb{R}$ and $b \in \mathbb{R} \cup \{\infty\}$. 
    We say that a map $g$ belongs to the \emph{Markov class of $I$}, denoted by $M(I)$, if:
    \begin{enumerate}
        \item There exists a sequence of closed intervals $(I_n)_{n \in \mathbb{N}}$ such that 
        \[
        \max(I_j) \leq \min(I_{j+1}), \quad |I_i \cap I_j| \leq 1 \ \text{for all } i,j \in \mathbb{N},
        \]
        \[
        \operatorname{dom}(g) = \bigcup_{n \in \mathbb{N}} I_n \subseteq I, \quad 
        \min(I_1) = \min(I), \quad \text{and } \sup(I_n) \to b \ \text{as } n \to \infty.
        \]
        
        \item The restriction $g_{|\mathring{I_i}}$ is injective, $\operatorname{Im}(g) \subseteq I$, and 
        \[
        g(\Gamma) \subseteq \Gamma, \quad \text{where } \Gamma = \bigcup_{n \in \mathbb{N}} \partial I_n.
        \]

        \item For all $i \in \mathbb{N}$, the set 
        \[
        g(I_i) \cap \left( \bigcup_{n \in \mathbb{N}} I_n \right)
        \]
        is a non-empty union of intervals of the form $I_n$.
    \end{enumerate}  
\end{definition}

\begin{remark}\label{emptyorall}
    From Condition (3) it follows that $g(I_i) \cap I_j$ is either empty, consists of a single point in $\Gamma$, or equals $I_j$. 
    Moreover, $g(I_i) \cap \mathring{I_j}$ is either empty or equal to $\mathring{I_j}$ for all $i,j \in \mathbb{N}$.
\end{remark}

\begin{remark}
    If $g: I \to I$ is a strictly monotone continuous function satisfying Conditions (1) and (3), 
    then Condition (2) is automatically fulfilled. 
    Some authors define Markov maps in different contexts under this assumption, 
    often together with additional regularity requirements (see \cite{OutroMarkov} and \cite{Rufus}). 
    We note that our definition differs slightly from those appearing in 
    \cite{CMP}, \cite{RMP5}, and \cite{RMP10}.
\end{remark}

We denote by $g_i$ the restriction of $g$ to $I_i$. 
By condition (1) above, the interval $I$ admits the following geometric description:

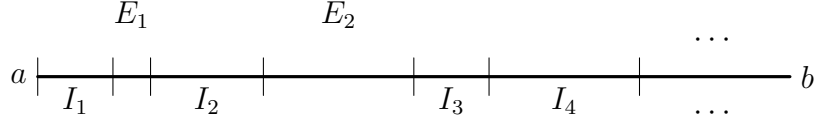
\begin{center}
\begin{tikzpicture}[x=1cm,y=1cm,line cap=round,line join=round]
  \draw[very thick] (0,0) -- (10,0);

  \foreach \x in {0,1,1.5,3,5,6,8} {
    \draw (\x,-0.25) -- (\x,0.25);
  }

  \node[left]  at (0,0)  {$a$};
  \node[right] at (10,0) {$b$};

  \node[below] at (0.5,0)  {$I_1$};
  \node[below] at (2.25,0) {$I_2$};
  \node[below] at (5.5,0)  {$I_3$};
  \node[below] at (7,0)    {$I_4$};

  \node[above] at (1.25,0.5) {$E_1$};
  \node[above] at (4,0.5)    {$E_2$};

  \node[below] at (9,-0.25) {$\cdots$};
  \node[above] at (9,0.25)  {$\cdots$};
\end{tikzpicture}
\end{center}

To make this picture precise, we write 
\[
I_1 = [c_0, c_1^{-}] 
\quad \text{and} \quad 
I_n = [c_{n-1}^{+}, c_{n}^{-}] \quad \text{for each } n \in \mathbb{N}.
\]
We then define the \textbf{escape sets} $E_n$ for each $n \in \mathbb{N}$ by
\[
E_n := (c_{n}^{-}, c_{n}^{+}).
\]
Note that $E_n$ may be empty for some values of $n$; this occurs precisely when 
$|I_n \cap I_{n+1}| = 1$. 
By construction, we have
\[
I = \left( \bigcup_{n \in \mathbb{N}} I_n \right) \cup \left( \bigcup_{n \in \mathbb{N}} E_n \right).
\]

If $g$ belongs to the Markov class of $I$, we associate to $g$ the matrix 
$A_g := (A_{ij})_{i,j \in \mathbb{N}}$, where
\[
A_{ij} =
\begin{cases}
    1 & \text{if } \mathring{I_j} \subseteq g(\mathring{I_i}), \\[4pt]
    0 & \text{otherwise}.
\end{cases}
\]

Let $\mathcal{G} = \mathcal{G}_g$ be the ultragraph induced by the matrix $A_g$, 
that is, 
\[
\mathcal{G} = (\mathcal{G}^0, \mathcal{G}^1, r, s),
\]
where
\[
\mathcal{G}^0 = \{ v_i \mid i \in \mathbb{N} \}, 
\quad 
\mathcal{G}^1 = \{ e_i \mid i \in \mathbb{N} \},
\]
\[
s(e_i) = v_i, 
\quad 
r(e_i) = \{ v_j \mid A_{ij} = 1 \} 
= \{ v_j \mid \mathring{I_j} \subseteq g(\mathring{I_i}) \}.
\]

The set
\[
E_g := \bigcup_{k \in \mathbb{N}} g^{-k} \left( \bigcup_{j \in \mathbb{N}} E_j \right)
\]
will play a crucial role in the next sections. 
Note that an element $x \in \operatorname{dom}(g)$ belongs to $E_g$ if and only if there exists 
$K \in \mathbb{N}$ such that $g^K(x) \notin \operatorname{dom}(g)$.

From now until the end of this chapter, let $g$ be a fixed Markov map and 
let $x \in E_g$ be a fixed element. 
We also write $\mathcal{G} = \mathcal{G}_g$ for the ultragraph induced by $g$, 
as defined above. 

Since $x \in E_g$, there exists a unique natural number $\tau(x)$ such that
\[
g^{\tau(x)-1}(x) \in \operatorname{dom}(g) 
\quad \text{and} \quad 
g^{\tau(x)}(x) \notin \operatorname{dom}(g),
\]
where by convention $g^0(x) = x$. 
We denote by $J = J_x$ the unique index such that 
$g^{\tau(x)}(x) \in E_J$, and define
\[
R_g(x) := 
\{ y \in \operatorname{dom}(g) \mid 
   g^n(y) = g^{\tau(x)}(x) \text{ for some } n \in \mathbb{N} \}.
\]


In the next lemma, we show that points in $R_g(x)$ cannot lie on the boundaries of the partition intervals, but must instead belong to their interiors. Moreover, we show that boundary points of $I_j$ never map into the interior of another interval.

\begin{lemma}\label{interior}
The following assertions hold:
\begin{enumerate}
    \item If $y \in R_g(x)$, then 
    $
    y \in \bigcup_{n \in \mathbb{N}} \mathring{I_n}.
    $

    \item If $\mathring{I_k} \subseteq g(I_j)$, then 
    $
    \mathring{I_k} \subseteq g(\mathring{I_j}).
    $
\end{enumerate}
\end{lemma}

\begin{proof}
\begin{enumerate}
    \item By hypothesis, there exists $n \in \mathbb{N}$ such that 
    \[
    g^n(y) = g^{\tau(x)}(x) \in E_{J_x} = E_J.
    \]
    Since $\Gamma \cap E_J = \emptyset$, we conclude that 
    $g^n(y) \notin \Gamma$. 
    By Condition (2) of Definition~\ref{Markovclass}, we know that 
    $g(\Gamma) \subseteq \Gamma$, and therefore $g^n(\Gamma) \subseteq \Gamma$. 
    Hence, if $y \in \Gamma$, then $g^n(y) \in \Gamma$, a contradiction. 
    Thus $y \notin \Gamma$, which proves the claim.

    \item Suppose $x \in \mathring{I_k}$. 
    By hypothesis, $x = g(x_1)$ for some $x_1 \in I_j$. 
    If $x_1 \in \partial I_j \subseteq \Gamma$, then $x = g(x_1) \in \Gamma$, 
    since $g(\Gamma) \subseteq \Gamma$. 
    But this contradicts the fact that 
    $\mathring{I_k} \cap \Gamma = \emptyset$. 
    Therefore $x_1 \in \mathring{I_j}$, and the result follows.
\end{enumerate}
\end{proof}

Our goal in the following pages is to construct representations of certain 
relative ultragraph algebras associated with $\mathcal{G} = \mathcal{G}_g$ 
using the Markov maps. 
We will build such representations in the $C^*$-algebra $B(\ell^2(R_g(x)))$. 
Recall that $\ell^2(R_g(x))$ is the space of square-summable functions 
$f:R_g(x) \to \mathbb{C}$ with respect to the counting measure. 
It has a canonical orthonormal basis given by the characteristic functions 
$\delta_y$ for $y \in R_g(x)$. 

For each $B \in \mathcal{E}$ and $i \in \mathbb{N}$ we define the following 
linear operators on the Hilbert basis of $\ell^2(R_g(x))$:
\[
W_B(\delta_y) 
   = \chi_{\left( \bigcup\limits_{\{j \mid v_j \in B \}} I_j \right)}(y)\,\delta_y,
\]
\[
T_{e_i}(\delta_y) 
   = \chi_{\left( \bigcup\limits_{\{j \mid v_j \in r(e_i) \}} I_j \right)}(y)\,
      \delta_{g_i^{-1}(y)}.
\]

It is straightforward to see that $W_B$ is well defined and extends to a 
bounded linear operator on $\ell^2(R_g(x))$. 
To justify that $T_{e_i}$ is well defined, observe that if 
$y \notin \bigcup_{\{j \mid v_j \in r(e_i)\}} I_j$, 
then $T_{e_i}(\delta_y) = 0$. 
Suppose instead that 
$y \in \bigcup_{\{j \mid v_j \in r(e_i)\}} I_j$. 
Then there exists $j \in \mathbb{N}$ with $v_j \in r(e_i)$ and $y \in I_j$. 
Consequently, $\mathring{I_j} \subseteq g(\mathring{I_i})$. 
Since $y \in R_g(x)$, Lemma~\ref{interior} implies $y \in \mathring{I_j}$. 
Thus $y = g(y_i)$ for a unique $y_i \in \mathring{I_i}$, because $g_i$ is injective. 
It follows that $g_i^{-1}(y) = y_i$ is well defined. 
Moreover, since $y \in R_g(x)$, the same holds for $g_i^{-1}(y)$, and hence 
$T_{e_i}$ extends to a bounded operator on $\ell^2(R_g(x))$.

\begin{proposition}
    If $B \in \mathcal{E}$ and $i \in \mathbb{N}$, then $  W_B^* = W_B$, $W_B^2 = W_B,$
    and for all $y \in R_g(x)$,
    \[
    T_{e_i}^*(\delta_y) 
       = \chi_{\left( \bigcup\limits_{\{j \mid v_j \in r(e_i)\}} I_j \right)}( g(y)) \,
         \chi_{I_i}(y)\,\delta_{g(y)}.
    \]
\end{proposition}

\begin{proof}
It is clear that $W_B^2 = W_B$. 
Let $y_1, y_2 \in R_g(x)$. 
A direct computation shows that
\[
\int \chi_{\left( \bigcup_{\{j \mid v_j \in B\}} I_j \right)}(y_1)\,
   \delta_{y_1}\,\delta_{y_2}\, d\mu
   = \int \delta_{y_1}\,
   \chi_{\left( \bigcup_{\{j \mid v_j \in B\}} I_j \right)}(y_2)\,
   \delta_{y_2}\, d\mu,
\]
and therefore $W_B^* = W_B$.  

To check that $T_{e_i}^*$ is well defined, note first that 
$T_{e_i}^*(\delta_y) = 0$ if 
\[
g(y) \notin \bigcup_{\{j \mid v_j \in r(e_i)\}} I_j.
\] 
Thus we may assume that 
$g(y) \in \bigcup_{\{j \mid v_j \in r(e_i)\}} I_j$. 
We claim that in this case $g(y) \in R_g(x)$.  

Indeed, since $y \in R_g(x)$, there exists $N \in \mathbb{N}$ such that 
$g^N(y) = g^{\tau(x)}(x) \in E_J$. 
If $N=1$, then $g(y) \in E_J$, which is impossible because
\[
E_J \cap \bigcup_{\{j \mid v_j \in r(e_i)\}} I_j = \emptyset.
\]
Hence $N \geq 2$, and consequently $g(y) \in R_g(x)$. 
This shows that $T_{e_i}^*$ is well defined, and one can verify that it extends to a bounded linear operator on $\ell^2(R_g(x))$.

Finally, to check the adjoint property, observe that
\[
\int \delta_{y_1}\,
   \chi_{\left( \bigcup_{\{j \mid v_j \in r(e_i)\}} I_j \right)}(g(y_2))\,
   \chi_{I_i}(y_2)\,
   \delta_{g(y_2)}\, d\mu = (*)
\] is equal to \[ (*)=
\begin{cases}
   0 & \text{if } y_2 \notin I_i, \\[4pt]
   0 & \text{if } y_2 \in I_i \text{ and } g_i(y_2) \neq y_1, \\[4pt]
   0 & \text{if } y_2 \in I_i,\, g_i(y_2) = y_1, 
       \text{ and } y_1 \notin \bigcup_{\{j \mid v_j \in r(e_i)\}} I_j, \\[4pt]
   1 & \text{if } y_2 \in I_i,\, g_i(y_2) = y_1, 
       \text{ and } y_1 \in \bigcup_{\{j \mid v_j \in r(e_i)\}} I_j.
\end{cases}
\]

It is straightforward to verify that the same holds for
\[
\int \chi_{\left( \bigcup_{\{j \mid v_j \in r(e_i)\}} I_j \right)}(y_1)\,
   \delta_{g_i^{-1}(y_1)}\,\delta_{y_2}\, d\mu.
\]
This proves the result.
\end{proof}

The next theorem shows how to construct a representation of the Toeplitz algebra using Markov maps.

\begin{theorem}\label{basic}
    The map 
    $
    \nu_x: C^*(\mathcal{G}, \emptyset) \;\longrightarrow\; B(\ell^2(R_g(x)))
    $
    defined by
    $
    \nu_x(p_B) = W_B$, $ 
    \nu_x(s_{e_i}) = T_{e_i},
    $
    is a $*$-homomorphism.
\end{theorem}

\begin{proof}
We check that $\{W_B, T_{e_i}\}_{B \in \mathcal{E},\, i \in \mathbb{N}}$ 
satisfies the relations that define $C^*(\mathcal{G}, \emptyset)$.  

First, $W_{\emptyset} = 0$.  
Since $R_g(x) \subseteq \bigcup_{i \in \mathbb{N}} \mathring{I_i}$ and 
$\mathring{I_i} \cap \mathring{I_j} = \emptyset$ for $i \neq j$, 
we obtain
\[
W_A \circ W_B = W_{A \cap B},
\qquad 
W_{A \cup B} = W_A + W_B - W_{A \cap B},
\]
which shows that the first item of Definition~\ref{relative} is satisfied.  

Next, for $y \in R_g(x)$ we compute:
\begin{align*}
     T_{e_i}^* \circ T_{e_i}(\delta_y) 
     &= T_{e_i}^* \Big( \chi_{\bigcup_{\{j \mid v_j \in r(e_i)\}} I_j}(y)\, 
                         \delta_{g_i^{-1}(y)} \Big) \\
     &= \chi_{\bigcup_{\{j \mid v_j \in r(e_i)\}} I_j}(y)\,
        \chi_{\bigcup_{\{j \mid v_j \in r(e_i)\}} I_j}(y)\,
        \chi_{I_i}(g_i^{-1}(y))\, \delta_y \\
     &= \chi_{\bigcup_{\{j \mid v_j \in r(e_i)\}} I_j}(y)\,\delta_y \\
     &= W_{r(e_i)}(\delta_y).
\end{align*}
Thus the second item holds.  
Applying $T_{e_i}$ to the equality above yields 
$T_{e_i} \circ T_{e_i}^* \circ T_{e_i} = T_{e_i}$, 
so $T_{e_i}$ is a partial isometry.  

Now, we compute:
\begin{align*}
    T_{e_i} \circ T_{e_i}^*(\delta_y) 
    &= T_{e_i}\Big(
       \chi_{\bigcup_{\{j \mid v_j \in r(e_i)\}} I_j}(g(y))\,
       \chi_{I_i}(y)\,\delta_{g(y)}
       \Big) \\
    &= \chi_{\bigcup_{\{j \mid v_j \in r(e_i)\}} I_j}(g(y))\,
       \chi_{I_i}(y)\,
       \chi_{\bigcup_{\{j \mid v_j \in r(e_i)\}} I_j}(g(y))\,
       \delta_y \\
    &= \chi_{\bigcup_{\{j \mid v_j \in r(e_i)\}} I_j}(g(y))\,
       \chi_{I_i}(y)\,\delta_y.
\end{align*}
Since
\[
W_{s(e_i)}(\delta_y) 
   = \chi_{\bigcup_{\{j \mid v_j \in s(e_i)\}} I_j}(y)\,\delta_y
   = \chi_{I_i}(y)\,\delta_y,
\]
we conclude that 
\[
W_{s(e_i)} \circ T_{e_i} \circ T_{e_i}^* 
   = T_{e_i} \circ T_{e_i}^*,
\]
which shows that the third item of Definition~\ref{relative} holds.  

Moreover, the computation above also shows that 
\[
T_{e_i} \circ T_{e_i}^* \circ T_{e_j} \circ T_{e_j}^*(\delta_y) = 0
\quad \text{for } i \neq j,
\]
so the partial isometries have orthogonal ranges.  

Finally, item (4) is vacuously satisfied since we chose $X = \emptyset$.  
\end{proof}

To obtain representations of relative $C^*$-algebras of the form 
$C^*(\mathcal{G},X)$ with $X \neq \emptyset$, 
we require an additional hypothesis. 
This assumption will be crucial throughout the remainder of this work 
and is introduced in the next lemma.

\begin{lemma}\label{lemaprincipalmarkov}
     If $i \in \mathbb{N}$ is such that 
     $\mathring{I_i} \cap g^{-1}(E_J) = \emptyset$, 
     then for all $y \in \mathring{I_i} \cap R_g(x)$ 
     there exists $j \in \mathbb{N}$ such that 
     $g(y) \in I_j$ and $v_j \in r(e_i)$.
\end{lemma}

\begin{proof}
Let $y \in \mathring{I_i} \cap R_g(x)$.  
By hypothesis, there exists $N \in \mathbb{N}$ such that 
\[
g^N(y) = g^{\tau(x)}(x) \in E_J.
\]
If $N=1$, then $g(y) \in E_J$, which would imply 
$y \in \mathring{I_i} \cap g^{-1}(E_J)$, a contradiction. 
Hence $N \geq 2$.  

Therefore $g(y) \in \operatorname{dom}(g)$, so we can choose 
$j \in \mathbb{N}$ with $g(y) \in I_j$. 
By Lemma~\ref{interior}, this implies $g(y) \in \mathring{I_j}$; 
in particular, $g(y) \in \mathring{I_j} \cap g(\mathring{I_i})$.

To conclude, it suffices to show that $v_j \in r(e_i)$.  
By the definition of the ultragraph $\mathcal{G}$, 
this is equivalent to proving that 
$\mathring{I_j} \subseteq g(\mathring{I_i})$.  
Indeed, since 
\[
g(y) \in \mathring{I_j} \cap g(\mathring{I_i}) 
   \subseteq \mathring{I_j} \cap g(I_i),
\] 
Remark~\ref{emptyorall} yields 
$\mathring{I_j} \cap g(I_i) = \mathring{I_j}$, 
and thus $\mathring{I_j} \subseteq g(I_i)$.  
The result then follows from Lemma~\ref{interior}.
\end{proof}

As a consequence of the above, under the additional hypothesis of 
    Lemma~\ref{lemaprincipalmarkov}, 
    the projection associated to a vertex $v_i$ decomposes as the sum of the 
    range projections of the edges emitted by $v_i$, and hence the representation constructed from the Markov map satisfies the full 
    set of Cuntz--Krieger relations. We state this precisely below.

\begin{corollary}\label{CK4}
    If $i \in \mathbb{N}$ is such that 
    $\mathring{I_i} \cap g^{-1}(E_J) = \emptyset$, 
    then
    \[
    W_{v_i} 
       = T_{e_i} \circ T_{e_i}^* 
       = \sum_{\,e \in \mathcal{G}^1 \,:\, s(e) = v_i} 
         T_{e} \circ T_{e}^*.
    \]
\end{corollary}

\begin{proof}
The second equality follows immediately from the definition of the ultragraph 
$\mathcal{G}$.  

For the first equality, let $y \in R_g(x)$. Then
\begin{align*}
    T_{e_i} \circ T_{e_i}^*(\delta_y) 
       &= \chi_{\bigcup_{\{j \mid v_j \in r(e_i)\}} I_j}(g(y))\,
          \chi_{I_i}(y)\,\delta_y \\
       &= \chi_{I_i}(y)\,\delta_y \\
       &= W_{v_i}(\delta_y).
\end{align*}
The second line follows from Lemma~\ref{lemaprincipalmarkov} when 
$y \in I_i$ (and thus $y \in \mathring{I_i}$, since $y \in R_g(x)$). 
If $y \notin I_i$, the equality is immediate. 
\end{proof}

\begin{proposition}\label{repMarkov}
    Let $X \subseteq \mathcal{G}^0$. 
    If for all $v_i \in X$ we have 
    $\mathring{I_i} \cap g^{-1}(E_J) = \emptyset$, 
    then the map $
    \nu_x: C^*(\mathcal{G}, X) \;\longrightarrow\; B(\ell^2(R_g(x))),
    $
    defined by
    $
    \nu_x(p_B) = W_B,$ $
    \nu_x(s_{e_i}) = T_{e_i},
    $
    is a $*$-homomorphism.
\end{proposition}

\begin{proof}
This follows directly from Theorem~\ref{basic} together with 
Corollary~\ref{CK4}.
\end{proof}

To conclude this section, we show that every representation $\nu_x$ constructed above coincides with a representation induced by a relative branching system. 
Hence, the class of such representations is contained in the class of representations induced by relative branching systems.  

Recall that $g \in M(I)$ and $x \in E_g$ are fixed. 
Also recall that the ultragraph $\mathcal{G}_g = \mathcal{G}$ induced by $g$ is the ultragraph defined by the adjacency matrix $A_g$, namely
\[
\mathcal{G}^0 = \{v_i \mid i \in \mathbb{N} \}, 
\qquad 
\mathcal{G}^1 = \{e_i \mid i \in \mathbb{N} \},
\]
with
\[
s(e_i) = v_i,
\qquad 
r(e_i) = \{v_j \mid A_{ij} = 1\} 
        = \{v_j \mid \mathring{I_j} \subseteq g(\mathring{I_i})\}.
\]

For each $j \in \mathbb{N}$ and for all $A \in \mathcal{E}$, define the sets
\[
D_{v_j} = I_j \cap R_g(x), 
\qquad 
D_A = \bigcup_{\{j \mid v_j \in A\}} D_{v_j}, 
\qquad 
R_{e_i} = \bigcup_{\{j \mid v_j \in r(e_i)\}} g_i^{-1}(I_j \cap R_g(x)).
\]

By Lemma~\ref{interior}, we have 
$I_j \cap R_g(x) = \mathring{I_j} \cap R_g(x)$.  
Therefore these sets can also be written as
\[
D_{v_j} = \mathring{I_j} \cap R_g(x), 
\qquad 
D_A = \bigcup_{\{j \mid v_j \in A\}} D_{v_j}, 
\qquad 
R_{e_i} = \bigcup_{\{j \mid v_j \in r(e_i)\}} g_i^{-1}(\mathring{I_j} \cap R_g(x)).
\]

Thus, depending on the context, we may choose whichever formulation of the sets 
is more convenient. 
Observe that all these sets are contained in $R_g(x)$ and therefore they are 
measurable with respect to the counting measure.  

For each $i \in \mathbb{N}$ we also define the function
\[
f_{e_i} \colon 
D_{r(e_i)} 
   = \bigcup_{\{j \mid v_j \in r(e_i)\}} \bigl(I_j \cap R_g(x)\bigr) 
   \;\longrightarrow\; 
R_{e_i} 
   = \bigcup_{\{j \mid v_j \in r(e_i)\}} g_i^{-1}(I_j \cap R_g(x))
\]
by $
f_{e_i}(y) = g_i^{-1}(y).
$
This construction yields a branching system, as stated in the next result.

\begin{proposition}\label{markovmapsinducesbs}
    Let $X \subseteq \mathcal{G}^0$. 
    If for all $v_i \in X$ we have 
    $\mathring{I_i} \cap g^{-1}(E_J) = \emptyset$, 
    then $
    \{D_A, R_{e_i}, f_{e_i}\}_{A \in \mathcal{E},\, i \in \mathbb{N}} $
    is a $(\mathcal{G},X)$-relative branching system in $\ell^2(R_g(x))$.
\end{proposition}

\begin{proof}
Let $A,B \in \mathcal{E}$. 
The equality $D_{A \cup B} = D_{A} \cup D_{B}$ is immediate. 
Moreover,
\[
D_{A} \cap D_{B} 
   = \bigcup_{\{j \mid v_j \in A\}}
       \Big( \bigcup_{\{i \mid v_i \in B\}} D_{v_j} \cap D_{v_i} \Big) 
   = \bigcup_{\{j \mid v_j \in A \cap B\}} D_{v_j} 
   = D_{A \cap B},
\]
since $D_{v_j} \cap D_{v_i} 
   = \mathring{I_i} \cap \mathring{I_j} \cap R_g(x)$, 
which equals $D_{v_j}$ if $i=j$ and is empty otherwise.  

Now let $x \in R_{e_i}$.  
Then $x \in I_i$. 
If $x \in \partial I_i$, then $g_i(x) \in \Gamma$.  
But since $x \in R_{e_i}$ we also have $g(x) \in R_g(x)$, 
and by Lemma~\ref{interior} it follows that $\Gamma \cap R_g(x) = \emptyset$, 
a contradiction.  
Thus $x \in \mathring{I_i}$, and therefore $R_{e_i} \subseteq \mathring{I_i}$ for all $i \in \mathbb{N}$. 
Consequently,
\[
R_{e_i} \cap R_{e_j} \subseteq \mathring{I_i} \cap \mathring{I_j} = \emptyset
\quad \text{for } i \neq j.
\]

To see that $R_{e_i} \subseteq D_{s(e_i)}$, observe first that 
\[
D_{s(e_i)} = D_{v_i} = I_i \cap R_g(x).
\] 
If $y \in R_{e_i}$, then clearly $y \in I_i$. 
Moreover, there exists $j \in \mathbb{N}$ with $v_j \in r(e_i)$ and 
$g(y) \in R_g(x) \cap I_j$. 
In particular, $y \in R_g(x)$. 
Hence $y \in I_i \cap R_g(x) = D_{v_i}$.  

We now prove that 
\[
D_{v_i} = \bigcup_{\{e \mid s(e) = v_i\}} R_e = R_{e_i}
\quad \text{for all } i \in \mathbb{N} \text{ with } v_i \in X.
\]
From the previous paragraph we know that $R_{e_i} \subseteq D_{v_i}$.  
Thus it remains to show $D_{v_i} \subseteq R_{e_i}$.  
Let $y \in D_{v_i}$.  
By our hypothesis, Lemma~\ref{lemaprincipalmarkov} guarantees that there exists 
$j \in \mathbb{N}$ such that $v_j \in r(e_i)$ and $g(y) \in I_j$.  

We claim that $g(y) \in R_g(x)$.  
Indeed, since $y \in D_{v_i}$, we know $y \in R_g(x)$, so there exists $N \in \mathbb{N}$ with
$
g^N(y) = g^{\tau(x)}(x) \in E_J.
$
If $N=1$, then $g(y) \in E_J$.  
But this contradicts the fact that $g(y) \in I_j$.  
Therefore $N \geq 2$, and consequently $g(y) \in R_g(x)$.  
Thus $
y \in g_i^{-1}(I_j \cap R_g(x)) \subseteq R_{e_i}$.

Finally, since we are working with the counting measure, to prove the last axiom 
of a relative branching system it suffices to show that each $f_{e_i}$ is a bijection.  
This follows directly from the definition, since $g_{|\mathring{I_i}}$ is injective.  
\end{proof}

We denote by $\pi_x$ the representation induced by the 
$(\mathcal{G},X)$-relative branching system above, via 
Proposition~\ref{relativerepinducedbybranchingsystems}. 
It follows that 
\[
\pi_x: C^*(\mathcal{G},X) \;\longrightarrow\; B(\ell^2(R_g(x)))
\]
is given by
\[
\pi_x(p_{A})(F) 
   = \chi_{\bigcup_{\{i \mid v_i \in A\}} (I_i \cap R_g(x))}\cdot F, 
   \qquad 
   \pi_x(s_{e_i})(F) 
   = \chi_{\bigcup_{\{j \mid v_j \in r(e_i)\}} g_i^{-1}(I_j \cap R_g(x))}\cdot (F \circ g_i),
\]
for all $F \in \ell^2(R_g(x))$.

\begin{theorem}\label{markovandbsareequal}
     Let $X \subseteq \mathcal{G}^0$. 
     If for all $v_i \in X$ we have $\mathring{I_i} \cap g^{-1}(E_J) = \emptyset$, 
     then the representations $\nu_x$ of Propositon~\ref{repMarkov} 
     and $\pi_x$ of the theorem above coincide.
\end{theorem}

\begin{proof}
It is enough to prove that $\pi_x = \nu_x$ on all generators of $C^*(\mathcal{G},X)$.  
Let $A \in \mathcal{E}$ and $i \in \mathbb{N}$. 
For $y,\tilde{y} \in R_g(x)$, a straightforward computation shows that
\[
\nu_x(p_A)(\delta_y)(\tilde{y}) =
\begin{cases}
    0 & \text{if } y \neq \tilde{y}, \\[4pt]
    0 & \text{if } y = \tilde{y} \text{ and } y \notin \bigcup_{\{j \mid v_j \in A\}} I_j, \\[4pt]
    1 & \text{if } y = \tilde{y} \text{ and } y \in \bigcup_{\{j \mid v_j \in A\}} I_j,
\end{cases}
= \pi_x(p_A)(\delta_y)(\tilde{y}).
\]

Similarly,
\[
\nu_x(s_{e_i})(\delta_y)(\tilde{y}) =
\begin{cases}
    0 & \text{if } g_i^{-1}(y) \neq \tilde{y}, \\[4pt]
    0 & \text{if } g_i^{-1}(y) = \tilde{y} 
         \text{ and } y \notin \bigcup_{\{j \mid v_j \in r(e_i)\}} I_j, \\[4pt]
    1 & \text{if } g_i^{-1}(y) = \tilde{y} 
         \text{ and } y \in \bigcup_{\{j \mid v_j \in r(e_i)\}} I_j,
\end{cases}
= \pi_x(s_{e_i})(\delta_y)(\tilde{y}).
\]

Therefore $\pi_x = \nu_x$ on all generators of $C^*(\mathcal{G},X)$, 
and the result follows.
\end{proof}


\section{Applications and Examples}

This section aims to present applications and concrete examples showing how
Markov maps, branching systems, and injectivity theorems can be applied in a
single setting.

Let $g \in M(I)$, let $x \in E_g$, and let $\mathcal{G}_g=\mathcal{G}$ be the
ultragraph induced by $g$. Fix $X \subseteq \mathcal{G}^0=\{v_1,v_2,\ldots\}$
such that, for every $v_i \in X$, we have
$\mathring{I_i}\cap g^{-1}(E_J)=\emptyset$.
By Theorem~\ref{markovandbsareequal}, the representation $\nu_x$ induced by the
Markov map coincides with the representation $\pi_x$ associated to the
$(\mathcal{G},X)$-relative branching system on $\ell^2(R_g(x))$ given by
\[
D_{v_i}=I_i\cap R_g(x),\qquad
R_{e_i}=\bigcup_{\{j\,:\, v_j\in r(e_i)\}} g_i^{-1}\big(I_j\cap R_g(x)\big),
\qquad
f_{e_i}=g_i^{-1}.
\]
As usual, write $Y=\mathrm{Reg}(\mathcal{G})\setminus X$; in our context we will
use $Y=\mathcal{G}^0\setminus X$ (since $\mathcal{G}$ has no sinks).

As the reader might expect, the injectivity theorems proved in the previous
sections apply to the representation $\nu_x=\pi_x$. For clarity, we record one
useful criterion and then apply it in the examples that follow.

\begin{theorem}\label{injectivityformarkovreps}
With the notation above, suppose that $r(e)\cap Y$ is finite for every
$e\in\mathcal{G}^1$. Then $\nu_x$ is injective if and only if:
\begin{enumerate}
    \item For every $i\in\mathbb{N}$ with $v_i\in X$,
          \,$I_i\cap R_g(x)\neq\emptyset$.
    \item For every $i\in\mathbb{N}$ with $v_i\in Y$ there exist
          $y_i\in I_i$ and $j\in\mathbb{N}$ with $v_j\in r(e_i)$ such that
          $g(y_i)\in I_j\cap R_g(x)$.
    \item For every $i\in\mathbb{N}$ with $v_i\in Y$,
          \,$I_i\cap R_g(x)\cap g^{-1}(E_J)\neq\emptyset$.
    \item For any cycle $\alpha=(e_{j_1},\ldots,e_{j_N})$ without exits such
          that $r(e_{j_i})\notin Y$ for all $i=1,\ldots,N$, and for any finite
          subset $\mathcal{F}\subseteq\mathbb{N}$, there exists
          $x\in D_{s(\alpha)}$ with
          \[
             \big(g_{j_1}^{-1}\circ\cdots\circ g_{j_N}^{-1}\big)^{\circ n}(x)\neq x
             \qquad \forall\, n\in\mathcal{F}.
          \]
\end{enumerate}
\end{theorem}

\begin{proof}
We prove that each of the hypotheses \((1)\)--\((4)\) is equivalent to one of
the hypotheses of Theorem~\ref{injectivitybscountingmeasure}. It is
straightforward that \((1)\) is equivalent to \((a)\) and \((2)\) is equivalent
to \((b)\) of Theorem~\ref{injectivitybscountingmeasure}. Since, in the
definition of the relative branching system, \(f_{e_i}=g_i^{-1}\), it follows
that \((4)\) is equivalent to \((d)\) of the same theorem.

It remains to show that \((3)\) is equivalent to the condition
\[
\Bigl(D_{v_i}\setminus \bigcup_{\{e\,:\,s(e)=v_i\}} R_e\Bigr)
= D_{v_i}\setminus R_{e_i}\neq \emptyset
\quad\text{for all } i\in\mathbb{N} \text{ with } v_i\in Y.
\]
First, fix \(i\in\mathbb{N}\) with \(v_i\in Y\) and suppose that
\(I_i\cap R_g(x)\cap g^{-1}(E_J)\neq\emptyset\). Then there exists
\(y_i\in I_i\cap R_g(x)\cap g^{-1}(E_J)\). In particular, \(y_i\in D_{v_i}\)
and \(g_i(y_i)\in E_J\). Therefore
\[
y_i\notin R_{e_i}
=\bigcup_{\{j\,:\, v_j\in r(e_i)\}} g_i^{-1}\bigl(R_g(x)\cap I_j\bigr),
\]
for otherwise \(g_i(y_i)\in R_g(x)\cap I_j\subseteq I_j\) for some
\(j\) with \(v_j\in r(e_i)\), which contradicts \(g_i(y_i)\in E_J\).
Hence \(D_{v_i}\setminus R_{e_i}\neq\emptyset\).

Conversely, suppose \(D_{v_i}\setminus R_{e_i}\neq\emptyset\) and take
\(y_i\in D_{v_i}=I_i\cap R_g(x)\) with
\[
y_i\notin \bigcup_{\{j\,:\, v_j\in r(e_i)\}} g_i^{-1}\bigl(R_g(x)\cap I_j\bigr).
\]
Choose \(N\in\mathbb{N}\) such that \(g^N(y_i)=g^{\tau(x)}(x)\).
If \(N=1\), then \(g(y_i)=g^{\tau(x)}(x)\in E_J\), and we are done.

If \(N\ge 2\), we obtain a contradiction as follows: from \(N\ge 2\) we have
\(g(y_i)\in R_g(x)\) and \(g(y_i)\in \bigcup_{j\in\mathbb{N}} I_j\). By
Lemma~\ref{interior}, there exists \(j_0\in\mathbb{N}\) such that
\(g(y_i)\in \mathring{I}_{j_0}\). By Remark~\ref{emptyorall} this implies
\(g(I_i)\cap \mathring{I}_{j_0}=\mathring{I}_{j_0}\), hence
\(\mathring{I}_{j_0}\subseteq g(I_i)\). By Lemma~\ref{interior},
\(\mathring{I}_{j_0}\subseteq g(\mathring{I}_i)\), and consequently
\(v_{j_0}\in r(e_i)\). Therefore
\[
y_i \in g_i^{-1}\bigl(I_{j_0}\cap R_g(x)\bigr)
\subseteq \bigcup_{\{j\,:\, v_j\in r(e_i)\}} g_i^{-1}\bigl(I_j\cap R_g(x)\bigr)
= R_{e_i},
\]
contradicting the choice of \(y_i\). Thus \(N\) cannot be greater or equal to 2, and the only
possibility is \(N=1\), which yields \(g(y_i)\in E_J\). Hence
\(I_i\cap R_g(x)\cap g^{-1}(E_J)\neq\emptyset\).

We have shown that \((3)\) is equivalent to
\(D_{v_i}\setminus R_{e_i}\neq\emptyset\) for all \(i\) with \(v_i\in Y\).
Combining this with the equivalences for \((1)\), \((2)\), and \((4)\), the
theorem follows.
\end{proof}

\begin{remark}
If $\mathcal{G}$ has no cycles without exits, or if $\mathcal{G}$ satisfies the relative Condition~(L), then item~(4) of the above theorem is vacuously satisfied.
\end{remark}

\begin{example}
Consider $I=[0,\infty)$. For each $n\in\mathbb{N}$, set
\[
I_n:=[2n-2,\,2n-1]\quad(n\ge 1),
\]
so that $E_n=(2n-1,\,2n)$. Define
\[
g(x)=
\begin{cases}
x-2, & x\in I_n \text{ with } n\ge 2,\\
3x,  & x\in I_1.
\end{cases}
\]
Note that if $n\ge 2$ then $g(I_n)=I_{n-1}$.

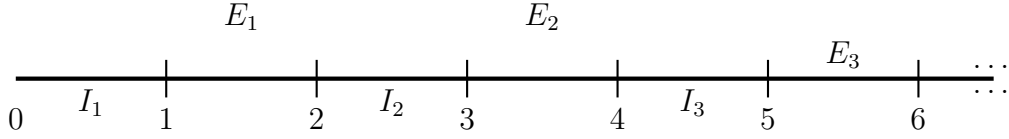
\begin{center}
\begin{tikzpicture}
    \draw[ultra thick,-] (-1,0) -- (12,0);
    \draw[thick] (1,-0.25) -- (1,0.25);
    \draw[thick] (3,-0.25) -- (3,0.25);
    \draw[thick] (5,-0.25) -- (5,0.25);
    \draw[thick] (7,-0.25) -- (7,0.25);
    \draw[thick] (9,-0.25) -- (9,0.25);
    \draw[thick] (11,-0.25) -- (11,0.25);
    \node[below] at (-1,-0.25) {$0$};
    \node[below] at (1,-0.25) {$1$};
    \node[below] at (3,-0.25) {$2$};
    \node[below] at (5,-0.25) {$3$};
    \node[below] at (7,-0.25) {$4$};
    \node[below] at (9,-0.25) {$5$};
    \node[below] at (11,-0.25) {$6$};
    \node[below] at (0,0) {$I_1$};
    \node[above] at (2,0.5) {$E_1$};
    \node[below] at (4,0) {$I_2$};
    \node[above] at (6,0.5) {$E_2$};
    \node[below] at (8,0) {$I_3$};
    \node[above] at (10,0) {$E_3$};
    \node[above] at (12,0) {$\dots$};
    \node[below] at (12,0) {$\dots$};
\end{tikzpicture}
\end{center}

We claim that $g\in M(I)$. Condition~(1) of Definition~\ref{Markovclass} is
clearly satisfied by the choice of the closed intervals $(I_n)_{n\in\mathbb{N}}$.
For condition~(2), note that $g|_{\mathring{I_i}}$ is injective for every $i$;
moreover,
\[
g(\partial I_1)=g(\{0,1\})=\{0,3\}\subseteq \Gamma
\quad\text{and}\quad
g(\partial I_n)=\partial I_{n-1}\subseteq \Gamma\ \ (n\ge 2),
\]
so $g(\Gamma)\subseteq \Gamma$, where $\Gamma=\bigcup_{n\in\mathbb{N}}\partial I_n$.
For condition~(3), if $n\ge 2$ we have $g(I_n)=I_{n-1}$, hence
$g(I_n)\cap\big(\bigcup_{m} I_m\big)=I_{n-1}$ is a union of intervals of the
form $I_m$. For $n=1$, $g(I_1)=[0,3]=I_1\cup E_1\cup I_2$, so
\[
g(I_1)\cap\bigcup_{m} I_m = I_1\cup I_2,
\]
again a union of intervals of the form $I_m$. Therefore $g$ satisfies all three
conditions of Definition~\ref{Markovclass}, and hence $g\in M(I)$.

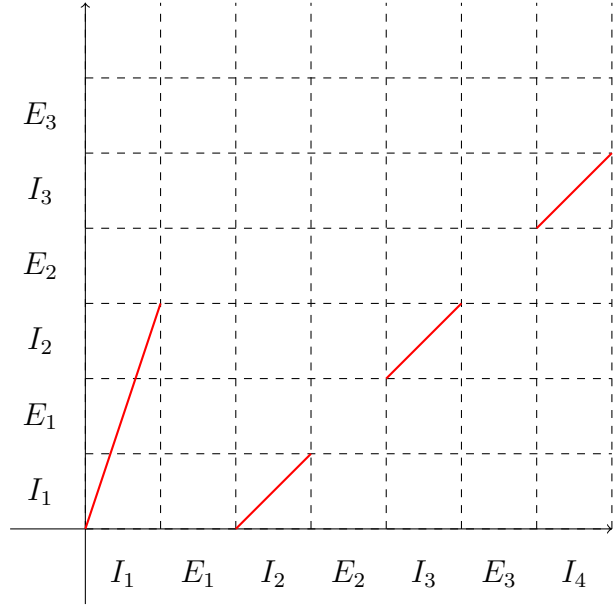
\begin{center}
\begin{tikzpicture}
  \draw[->] (-1,0) -- (7,0);
  \draw[->] (0,-1) -- (0,7);
  \foreach \i in {1,2,3} \node at ({2*\i - 0.5}, -0.6) {$E_{\i}$};
  \foreach \i in {1,2,3,4} \node at ({2*\i - 1.5}, -0.6) {$I_{\i}$};
  \foreach \i in {1,2,3} \node at (-0.6, {2*\i - 0.5}) {$E_{\i}$};
  \foreach \i in {1,2,3} \node at (-0.6, {2*\i - 1.5}) {$I_{\i}$};
  \foreach \x in {0,1,2,3,4,5,6,7} \draw[dashed] (\x,0) -- (\x,7);
  \foreach \y in {0,1,2,3,4,5,6} \draw[dashed] (0,\y) -- (7,\y);
  \draw[red,thick] (0,0) -- (1,3);
  \draw[red,thick] (2,0) -- (3,1);
  \draw[red,thick] (4,2) -- (5,3);
  \draw[red,thick] (6,4) -- (7,5);
\end{tikzpicture}
\end{center}

Recall that
\[
r(e_i)=\{\,v_j \mid \mathring{I_j}\subseteq g(\mathring{I_i})\,\},
\]
and by Lemma~\ref{interior}, if $\mathring{I_j}\subseteq g(I_i)$ then
$\mathring{I_j}\subseteq g(\mathring{I_i})$. In the present example,
$g(I_1)=I_1\cup E_2\cup I_2$ and $g(I_n)=I_{n-1}$ for $n\ge 2$, hence
$
r(e_1)=\{v_1,v_2\}$ and $r(e_n)=\{v_{n-1}\}\ \ (n\ge 2)$.
Thus the ultragraph $\mathcal{G}$ induced by $g$ has the following description:
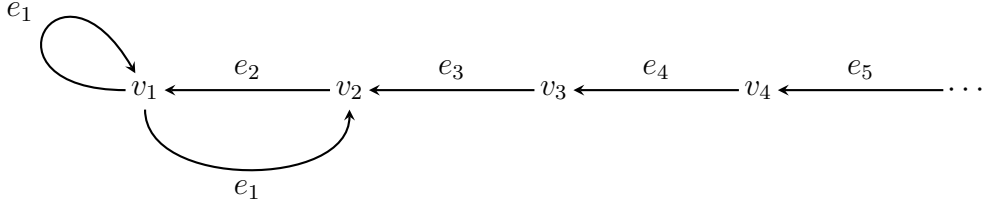
\begin{center}
\begin{tikzpicture}[>=stealth,->,auto,node distance=2.2cm,thick]
  \usetikzlibrary{automata,positioning}
  \tikzset{every state/.style={minimum size=0pt}}
  \tikzset{every loop/.style={min distance=10mm,in=0,out=80,looseness=20}}

  \node[state,inner sep=0.5pt,draw=none] (A)                {$v_1$};
  \node[state,inner sep=0.5pt,draw=none,right=of A] (B)     {$v_2$};
  \node[state,inner sep=0.5pt,draw=none,right=of B] (C)     {$v_3$};
  \node[state,inner sep=0.5pt,draw=none,right=of C] (D)     {$v_4$};
  \node[state,inner sep=0.5pt,draw=none,right=of D] (E)     {$\cdots$};

  \path
    (B) edge node[above] {$e_2$} (A)
    (C) edge node[above] {$e_3$} (B)
    (D) edge node[above] {$e_4$} (C)
    (E) edge node[above] {$e_5$} (D)
    (A) edge [in=120,out=180,looseness=20] node {$e_1$} (A)
    (A) edge [in=270,out=270] node[below] {$e_1$} (B);
\end{tikzpicture}
\end{center}

Consider the point $x=\tfrac{1}{2}$. Note that $x\in E_g$ because
$g(x)=\tfrac{3}{2}\in E_1$. It follows that $\tau(x)=1$ and $J=J_x=1$.
Let $X=\{\,v_i \mid i\ge 2\,\}$. Since $g^{-1}(E_J)=\bigl(\tfrac{1}{3},\tfrac{2}{3}\bigr)$,
we have $\mathring{I_i}\cap g^{-1}(E_J)=\emptyset$ for all $i\ge 2$.
By Propositon~\ref{repMarkov} there is a representation
\[
\nu_{1/2}: C^*(\mathcal{G},X)\longrightarrow B\bigl(\ell^2(R_g(1/2))\bigr),
\]
and by Theorem~\ref{markovandbsareequal} we have $\nu_{1/2}=\pi_{1/2}$. We claim
that $\pi_{1/2}$ is injective. To verify this, we check the hypotheses of
Theorem~\ref{injectivityformarkovreps}:

\begin{enumerate}
    \item For each $i\ge 2$, let $y_i$ be the midpoint of $I_i$. Then
          $g(y_i)=y_i-2$, which is the midpoint of $I_{i-1}$, and iterating we
          obtain $g^{\,i-1}(y_i)=\tfrac{1}{2}$ and hence
          $g^{\,i}(y_i)=\tfrac{3}{2}=g^{\tau(x)}(x)$. Therefore
          $y_i\in I_i\cap R_g(x)$.
    \item For $i=1$, choose $y_1=\tfrac{1}{6}\in I_1$ and set $j=1$ (note that
          $v_1\in r(e_1)$). Then $g(\tfrac{1}{6})=\tfrac{1}{2}\in I_1\cap R_g(x)$.
    \item Observe that $\tfrac{1}{2}\in I_1\cap R_g(x)\cap g^{-1}(E_1)$.
    \item This item holds vacuously.
\end{enumerate}

\end{example}

\begin{example}
Consider the interval $I = [0,3)$ and define $I_1 = [0,1]$, $I_2 = [2, 2 + \frac{1}{2}]$, and $$I_n = [2 + \frac{1}{2} + \ldots \frac{1}{2^{n - 2}}, 2 + \frac{1}{2} + \ldots \frac{1}{2^{n - 1}}], \text{ for all $n \geq 3$.}$$ Define $g: \bigcup\limits_{n \in \mathbb{N}} I_n \to I$ as $g(x) = 2x + 1$ if $x \in (0,1)$; $g(0) = g(1) = 2$, and for $n \geq 2$ the map $g|_{I_n}$ is the crescent linear homeomorphism between $I_n$ and $I_{n - 1}$. Thus $g(I_1)=(1,3)$ and $g(I_n)=I_{n-1}$ for all $n\ge 2$. It follows that
$g\in M(I)$ and $E_n=\emptyset$ for all $n\ge 2$.
 
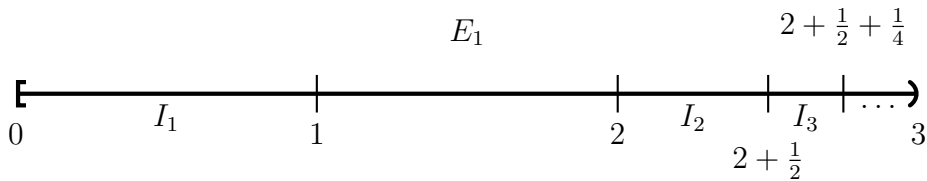
\begin{center}
\begin{tikzpicture}
  \draw[ultra thick,{[-) }] (-1,0) -- (11,0);
  \draw[thick] (3,-0.25) -- (3,0.25);
  \draw[thick] (7,-0.25) -- (7,0.25);
  \draw[thick] (9,-0.25) -- (9,0.25);
  \draw[thick] (10,-0.25) -- (10,0.25);
  \node[below] at (-1,-0.25) {$0$};
  \node[below] at (3,-0.25) {$1$};
  \node[below] at (7,-0.25) {$2$};
  \node[below] at (9,-0.5) {$2 + \tfrac{1}{2}$};
  \node[above] at (10,0.5) {$2 + \tfrac{1}{2} + \tfrac{1}{4}$};
  \node[below] at (11,-0.25) {$3$};
  \node[below] at (1,0) {$I_1$};
  \node[above] at (5,0.5) {$E_1$};
  \node[below] at (8,0) {$I_2$};
  \node[below] at (9.5,0) {$I_3$};
  \node[below] at (10.5,0) {$\dots$};
\end{tikzpicture}
\end{center}

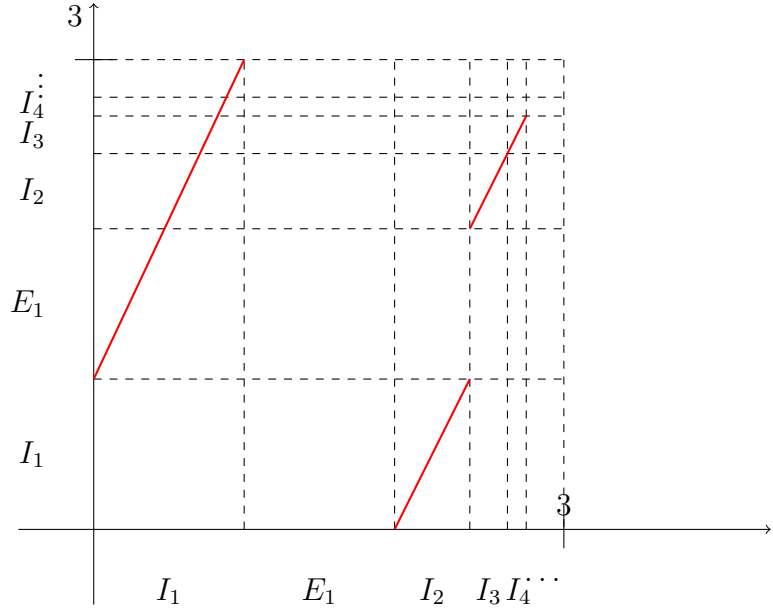
\begin{center}
\begin{tikzpicture}
  \draw[->] (-1,0) -- (9,0);
  \draw[->] (0,-1) -- (0,7);
  \node[below] at (1,-0.5) {$I_1$};
  \node[below] at (3,-0.5) {$E_1$};
  \node[below] at (4.5,-0.5) {$I_2$};
  \node[below] at (5.25,-0.5) {$I_3$};
  \node[below] at (5.65,-0.5) {$I_4$};
  \node[below] at (6,-0.5) {$\dots$};
  \node[left] at (-0.5, 1) {$I_1$};
  \node[left] at (-0.5, 3) {$E_1$};
  \node[left] at (-0.5, 4.5) {$I_2$};
  \node[left] at (-0.5, 5.25) {$I_3$};
  \node[left] at (-0.5, 5.65) {$I_4$};
  \node[left] at (-0.5, 6) {$\vdots$};
  \foreach \y in {2,4} { \draw[dashed] (0,\y) -- (6.25,\y); }
  \draw[-] (0.25, 6.25) -- (-0.25, 6.25) node[left= 0.3cm,above = 0.3cm] {$3$};
  \draw[dashed] (0,5) -- (6.25,5);
  \draw[dashed] (0,5.5) -- (6.25,5.5);
  \draw[dashed] (0,5.75) -- (6.25,5.75);
  \draw[dashed] (6.25, 6.25) -- (6.25,0);
  \foreach \x in {2,4} { \draw[dashed] (\x,0) -- (\x,6.25); }
  \draw[-] (6.25,0.25) -- (6.25, -0.25) node[right = 0.3cm,above = 0.3cm] {$3$};
  \draw[dashed] (5,0) -- (5,6.25);
  \draw[dashed] (5.5,0) -- (5.5,6.25);
  \draw[dashed] (5.75,0) -- (5.75,6.25);
  \draw[dashed] (0, 6.25) -- (6.25, 6.25);
  \draw[red,thick] (4,0) -- (5,2);
  \draw[red,thick] (0,2) -- (2,6.25);
  \draw[red,thick] (5,4) -- (5.5,5);
  \draw[red,thick] (5.5,5) -- (5.75,5.5);
\end{tikzpicture}
\end{center}

The ultragraph $\mathcal{G}$ induced by $g$ is geometrically the following:
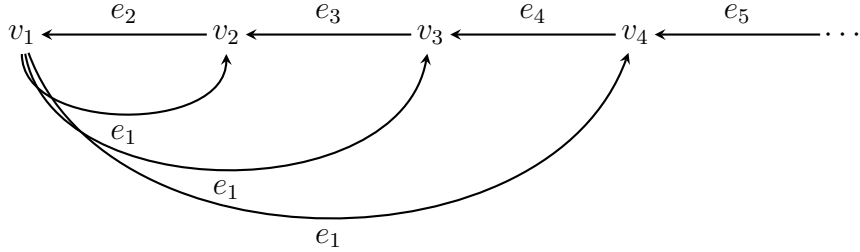
\begin{center}
\begin{tikzpicture}[>=stealth,->,auto,node distance=2.2cm,thick]
  \tikzset{every state/.style={minimum size=0pt}}
  \tikzset{every loop/.style={min distance=10mm,in=0,out=80,looseness=20}}

  \node[state,inner sep=0.5pt,draw=none] (A)                {$v_1$};
  \node[state,inner sep=0.5pt,draw=none,right=of A] (B)     {$v_2$};
  \node[state,inner sep=0.5pt,draw=none,right=of B] (C)     {$v_3$};
  \node[state,inner sep=0.5pt,draw=none,right=of C] (D)     {$v_4$};
  \node[state,inner sep=0.5pt,draw=none,right=of D] (E)     {$\cdots$};

  \path
    (B) edge node[above] {$e_2$} (A)
    (C) edge node[above] {$e_3$} (B)
    (D) edge node[above] {$e_4$} (C)
    (E) edge node[above] {$e_5$} (D);

  \path
    (A) edge[in=270,out=270]              node[below] {$e_1$} (B)
    (A) edge[in=260,out=280,looseness=1]  node[below] {$e_1$} (C)
    (A) edge[in=250,out=290,looseness=1]  node[below] {$e_1$} (D);
\end{tikzpicture}
\end{center}
Indeed, $
r(e_1)=\{v_j\mid j\ge 2\}=\{v_j\mid \mathring{I_j}\subseteq g(\mathring{I_1})\}$, and
$
r(e_n)=\{v_{n-1}\}=\{v_j\mid \mathring{I_j}\subseteq g(\mathring{I_n})\},\ n\ge 2.$

Consider the point $x=\tfrac{17}{8}$. Note that
$g^2(\tfrac{17}{8})=g(\tfrac14)=\tfrac32\in E_1$. Hence
$x\in E_g$, $\tau(x)=2$, and $J=J_x=1$. Let $X=\{v_i\mid i\ge 2\}$. Since
$g^{-1}(E_J)=g^{-1}(E_1)=(0,\tfrac12)$, we have
$g^{-1}(E_1)\cap \mathring{I_i}=\emptyset$ for all $i\ge 2$. By
Propositon~\ref{repMarkov} the representation
\[
\nu_{\frac{17}{8}}:\ C^*(\mathcal{G},X)\longrightarrow B\bigl(\ell^2(R_g(\tfrac{17}{8}))\bigr)
\]
is well defined. By Theorem~\ref{markovandbsareequal} we have
$\nu_{\frac{17}{8}}=\pi_{\frac{17}{8}}$. If the hypotheses of
Theorem~\ref{injectivityformarkovreps} are verified, then
$\nu_{\frac{17}{8}}$ is injective. We check them:

\begin{itemize}
  \item[(1)] For each $i\ge 2$, choose $y_i\in I_i$ such that
             $g^{\,i-1}(y_i)=\tfrac14$ (this exists and is unique since
             $g^{\,i-1}:I_i\to I_1$ is a linear homeomorphism). Then
             $g^{\,i}(y_i)=g(\tfrac14)=\tfrac32=g^{\tau(x)}(x)$, so
             $y_i\in I_i\cap R_g(x)$.
  \item[(2)] For $i=1$, take $y_1=\tfrac{9}{16}\in I_1$ and $j=2$ (note that
             $v_2\in r(e_1)$). Since $g(\tfrac{9}{16})=\tfrac{17}{8}\in I_2\cap R_g(x)$,
             the condition holds.
  \item[(3)] We have $\tfrac14\in I_1\cap R_g(\tfrac{17}{8})\cap g^{-1}(E_1)$.
  \item[(4)] Holds vacuously (there is no cycle without exits satisfying the
             stated condition with range in $Y$).
\end{itemize}
Therefore, the representation $\nu_{\frac{17}{8}}$ is injective.
\end{example}

\begin{remark}
In the example above, if we choose $X'$ to be any proper subset of $\{v_i : i \ge 2\}$, the representation
\[
\nu_{17/8} \colon C^*(\mathcal{G},X') \longrightarrow B\bigl(\ell^2(R_g(\tfrac{17}{8}))\bigr)
\]
still exists, since $g^{-1}(E_1)=(0,\tfrac{1}{2})$ and therefore $\mathring{I_i}\cap g^{-1}(E_1)=\varnothing$ for all $i\ge 2$. However, the conclusion changes: $\nu_{17/8}$ is \emph{not} injective. Indeed, if $j\ge 2$ is such that $v_j\notin X'$ (so $v_j\in Y:=\mathcal{G}^0\setminus X'$), then
\[
I_j \cap R_g(\tfrac{17}{8}) \cap g^{-1}(E_1) \;=\; \varnothing,
\]
because $I_j \cap g^{-1}(E_1)=\varnothing$ for all $j\ge 2$. The same reasoning applies, \emph{mutatis mutandis}, to the preceding example.
\end{remark}

We finish this work with an example where $Y=\mathrm{Reg}(\mathcal{G})\setminus X$ is an infinite set.

\begin{example}
Consider the interval $I=[0,\infty)$. Define $I_1=[0,1]$, $I_2=[2,3]$, and
$I_n=[\,n+1,\,n+2\,]$ for all $n\ge 3$.
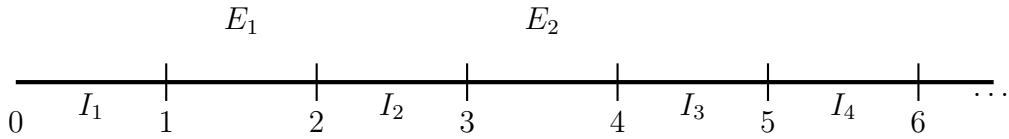
\begin{center}
\begin{tikzpicture}
  \draw[ultra thick,-] (-1,0) -- (12,0);
  \draw[thick] (1,-0.25) -- (1,0.25);
  \draw[thick] (3,-0.25) -- (3,0.25);
  \draw[thick] (5,-0.25) -- (5,0.25);
  \draw[thick] (7,-0.25) -- (7,0.25);
  \draw[thick] (9,-0.25) -- (9,0.25);
  \draw[thick] (11,-0.25) -- (11,0.25);
  \node[below] at (-1,-0.25) {$0$};
  \node[below] at (1,-0.25) {$1$};
  \node[below] at (3,-0.25) {$2$};
  \node[below] at (5,-0.25) {$3$};
  \node[below] at (7,-0.25) {$4$};
  \node[below] at (9,-0.25) {$5$};
  \node[below] at (11,-0.25) {$6$};
  \node[below] at (0,0) {$I_1$};
  \node[above] at (2,0.5) {$E_1$};
  \node[below] at (4,0) {$I_2$};
  \node[above] at (6,0.5) {$E_2$};
  \node[below] at (8,0) {$I_3$};
  \node[below] at (10,0) {$I_4$};
  \node[below] at (12,0) {$\dots$};
\end{tikzpicture}
\end{center}

Note that $E_1=(1,2)$, $E_2=(3,4)$, and $E_n=\emptyset$ for all $n\ge 3$. We define $g$ as follows.

\begin{itemize}
\item For $x\in I_1$: if $x\in \mathring{I_1}$ set $g(x)=2x+1$, and if $x\in\{0,1\}$ define $g(x)=3$.
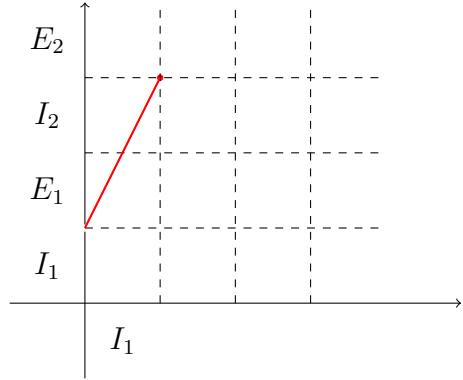
\begin{center}
\begin{tikzpicture}
  \draw[->] (-1,0) -- (5,0);
  \draw[->] (0,-1) -- (0,4);
  \node at (0.5,-0.5) {$I_1$};
  \node at (-0.5, 0.5) {$I_1$};
  \node at (-0.5, 1.5) {$E_1$};
  \node at (-0.5, 2.5) {$I_2$};
  \node at (-0.5, 3.5) {$E_2$};
  \filldraw[red] (1,3) circle (1pt);
  \filldraw[white,thick] (0,1) circle (1pt);
  \foreach \x in {1,2,3} \draw[dashed] (\x,0) -- (\x,4);
  \foreach \y in {1,2,3} \draw[dashed] (0,\y) -- (4,\y);
  \draw[red,thick] (0,1) -- (1,3);
\end{tikzpicture}
\end{center}

\item If $n\ge 2$ is even, divide $I_n=[\,n+1,\,n+2\,]$ into three equal subintervals
\[
A_1^n=[\,n+1,\,n+1+\tfrac{1}{3}\,],\quad
A_2^n=(\,n+1+\tfrac{1}{3},\,n+1+\tfrac{2}{3}\,),\quad
A_3^n=[\,n+1+\tfrac{2}{3},\,n+2\,].
\]
Define $g$ on $A_1^n$ as the increasing linear homeomorphism onto $I_{n-1}$,
on $A_3^n$ as the increasing linear homeomorphism onto $I_{n+1}$, and on
$A_2^n$ as the increasing linear homeomorphism onto $E_2$.
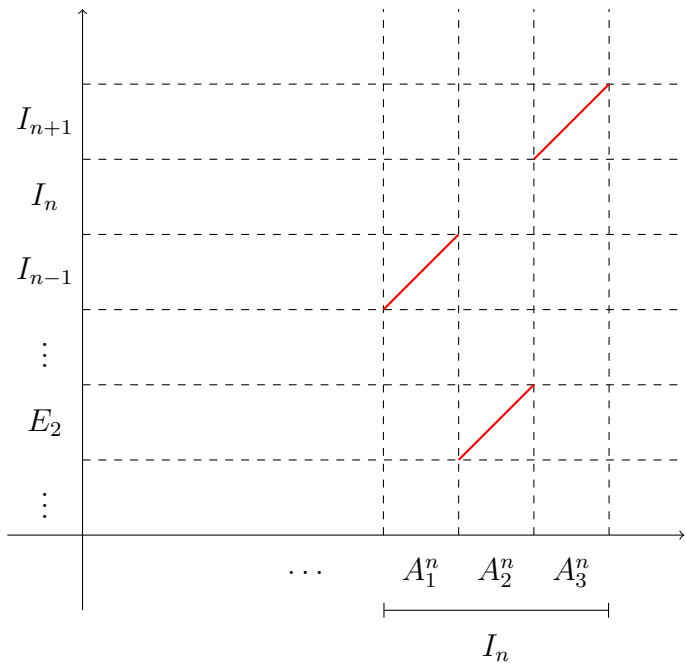
\begin{center}
\begin{tikzpicture}
  \draw[->] (-1,0) -- (8,0);
  \draw[->] (0,-1) -- (0,7);
  \draw[|-|] (4,-1) -- (7,-1);
  \node at (5.5,-1.5) {$I_n$};
  \node at (4.5,-0.5) {$A_1^n$};
  \node at (5.5,-0.5) {$A_2^n$};
  \node at (6.5,-0.5) {$A_3^n$};
  \node at (-0.5,3.5) {$I_{n-1}$};
  \node at (-0.5,4.5) {$I_n$};
  \node at (-0.5,5.5) {$I_{n+1}$};
  \node at (3,-0.5) {$\dots$};
  \node at (-0.5,1.5) {$E_2$};
  \node at (-0.5,2.5) {$\vdots$};
  \node at (-0.5,0.5) {$\vdots$};
  \foreach \x in {4,5,6,7} \draw[dashed] (\x,0) -- (\x,7);
  \foreach \y in {1,2,3,4,5,6} \draw[dashed] (0,\y) -- (8,\y);
  \draw[red,thick] (4,3) -- (5,4);
  \draw[red,thick] (5,1) -- (6,2);
  \draw[red,thick] (6,5) -- (7,6);
\end{tikzpicture}
\end{center}

\item If $n\ge 3$ is odd, divide $I_n=[\,n+1,\,n+2\,]$ similarly into
$A_1^n$, $A_2^n$, $A_3^n$. Define $g$ on $A_1^n$ as the increasing linear
homeomorphism onto $I_{n-1}$, on $A_3^n$ as the increasing linear homeomorphism
onto $I_{n+1}$, and on $A_2^n$ as the increasing linear homeomorphism onto $E_1$.
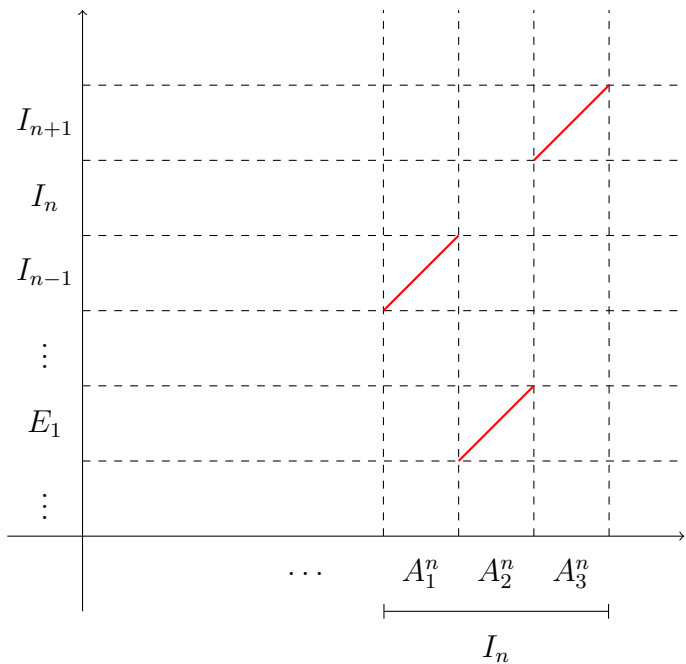
\begin{center}
\begin{tikzpicture}
  \draw[->] (-1,0) -- (8,0);
  \draw[->] (0,-1) -- (0,7);
  \draw[|-|] (4,-1) -- (7,-1);
  \node at (5.5,-1.5) {$I_n$};
  \node at (4.5,-0.5) {$A_1^n$};
  \node at (5.5,-0.5) {$A_2^n$};
  \node at (6.5,-0.5) {$A_3^n$};
  \node at (-0.5,3.5) {$I_{n-1}$};
  \node at (-0.5,4.5) {$I_n$};
  \node at (-0.5,5.5) {$I_{n+1}$};
  \node at (3,-0.5) {$\dots$};
  \node at (-0.5,1.5) {$E_1$};
  \node at (-0.5,2.5) {$\vdots$};
  \node at (-0.5,0.5) {$\vdots$};
  \foreach \x in {4,5,6,7} \draw[dashed] (\x,0) -- (\x,7);
  \foreach \y in {1,2,3,4,5,6} \draw[dashed] (0,\y) -- (8,\y);
  \draw[red,thick] (4,3) -- (5,4);
  \draw[red,thick] (5,1) -- (6,2);
  \draw[red,thick] (6,5) -- (7,6);
\end{tikzpicture}
\end{center}
\end{itemize}

The construction for even and odd indices is the same except for the target
escape set ($E_2$ vs.\ $E_1$). Condition~(1) in the definition of Markov maps
holds immediately. Conditions~(2) and~(3) also follow from the definitions,
since on each subinterval we use increasing linear homeomorphisms and split
each $I_n$ into three parts. In particular,
\[
g(I_1)=E_1\cup I_2,\qquad
g(I_n)=I_{n-1}\cup I_{n+1}\cup E_1\ \ \text{if $n$ is odd,}
\]
\[
g(I_n)=I_{n-1}\cup I_{n+1}\cup E_2\ \ \text{if $n\ge 2$ is even.}
\]

From the behavior of $g$ on the interiors, it follows that
$r(e_1)=\{v_2\}$ and $r(e_n)=\{v_{n-1},v_{n+1}\}$ for all $n\ge 2$.
Hence, the ultragraph $\mathcal{G}$ induced by $g$ is:
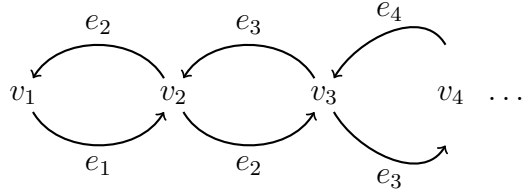
\begin{center}
\begin{tikzpicture}[->,auto,node distance=2cm,thick]
  \tikzset{every state/.style={minimum size=0pt}}
  \tikzset{every loop/.style={min distance=10mm,in=0,out=80,looseness=20}}
  \node[state,inner sep=0.5pt,draw=none] (A) {$v_1$};
  \node[state,inner sep=0.5pt,draw=none,right of=A] (B) {$v_2$};
  \node[state,inner sep=0.5pt,draw=none,right of=B] (C) {$v_3$};
  \node[state,inner sep=0.5pt,draw=none,right of=C] (D) {$\;\;v_4\;\;\dots$};
  \path
    (B) edge [in=60,out=120] node[above] {$e_2$} (A)
    (A) edge [in=240,out=300] node[below] {$e_1$} (B)
    (C) edge [in=60,out=120] node[above] {$e_3$} (B)
    (B) edge [in=240,out=300] node[below] {$e_2$} (C)
    (D) edge [in=60,out=120] node[above] {$e_4$} (C)
    (C) edge [in=240,out=300] node[below] {$e_3$} (D);
\end{tikzpicture}
\end{center}

Now take $x=\tfrac{5}{2}\in I_2$. Then $g(x)=\tfrac{7}{2}\in E_2$, so
$x\in E_g$, $\tau(x)=1$, and $J=J_x=2$. Let $X=\{\,v_i \mid i \text{ is odd}\,\}$.
Note that $I_i\cap g^{-1}(E_2)=\emptyset$ for all odd $i$, because
$g(I_i)=I_{i-1}\cup I_{i+1}\cup E_1$ when $i$ is odd. Therefore the
representation $\nu_x=\pi_x\colon C^*(\mathcal{G},X)\to B(\ell^2(R_g(x)))$
is well defined. Using Theorem~\ref{injectivityformarkovreps}, we show that
$\nu_x$ is injective:

\begin{enumerate}
  \item We must show $I_i\cap R_g(x)\neq\emptyset$ for every odd $i$.
        For $i=1$, since $g(I_1)=E_1\cup I_2$ and $\tfrac{5}{2}\in I_2$,
        pick $y_1\in I_1$ with $g(y_1)=\tfrac{5}{2}$; then
        $g^2(y_1)=g^{\tau(x)}(x)=\tfrac{7}{2}$.  
        For $i=3$, note $g(I_3)=I_2\cup I_4\cup E_1$; choose
        $y_2\in I_2$ with $g(y_2)=\tfrac{7}{2}$, and then $y_3\in I_3$
        with $g(y_3)=y_2$. Hence $g^2(y_3)=\tfrac{7}{2}$, so
        $y_3\in I_3\cap R_g(x)$.  
        For $i=5$, use $g(I_5)=I_4\cup I_6\cup E_1$ and
        $g(I_4)=I_3\cup I_5\cup E_2$ to choose $y_4\in I_4$ with
        $g(y_4)=\tfrac{7}{2}$ and then $y_5\in I_5$ with $g(y_5)=y_4$.
        Thus $y_5\in I_5\cap R_g(x)$. The same argument applies to all odd $i$.
  \item Let $i$ be even. Since $i-1$ is odd, by (1) there exists
        $y_{i-1}\in I_{i-1}\cap R_g(x)$. As $g(I_i)=I_{i-1}\cup I_{i+1}\cup E_2$,
        we can write $y_{i-1}=g(y_i)$ for some $y_i\in I_i$. Since
        $v_{i-1}\in r(e_i)$, the condition holds.
  \item Let $i$ be even. Because $g(I_i)=I_{i-1}\cup I_{i+1}\cup E_2$, choose
        $y_i\in I_i$ with $g(y_i)=\tfrac{7}{2}\in E_2$. Then
        $y_i\in I_i\cap R_g(x)\cap g^{-1}(E_2)$.
  \item This item is vacuous.
\end{enumerate}
Therefore, $\nu_x=\nu_{5/2}$ is injective.
\end{example}




\medskip

\section{Declarations}

\subsection*{Ethical Approval:}

This declaration is not applicable.

\subsection*{Conflicts of interests/Competing interests:} We have no conflicts of interests/competing interests to disclose.

\subsection*{Authors' contributions:}

All authors contributed equally to this work.

\subsection*{Funding:}
The first author would like to thank do Fundação de Amparo à Pesquisa e Inovação do Estado de Santa Catarina (FAPESC) for the essential financial support. The second author was partially supported by Capes-PrInt, Conselho Nacional de Desenvolvimento Cient\'ifico e Tecnol\'ogico (CNPq) - Brazil, and Funda\c{c}\~ao de Amparo \`a Pesquisa e Inova\c{c}\~ao do Estado de Santa Catarina (FAPESC).

\subsection*{Data Availability and materials:} Data sharing not applicable to this article as no datasets were generated or analysed during the current study.






\begin{thebibliography}{99}

\bibitem{AK}
M.\ Abe and K.\ Kawamura,
\textit{Recursive fermion system in Cuntz algebra. I. Embeddings of fermion algebra into Cuntz algebra},
Comm.\ Math.\ Phys.\ \textbf{228} (2002), 85--101.

\bibitem{Livroversaoalgebrica}
G.\ Abrams, P.\ Ara and M.\ Siles Molina,
\textit{Leavitt Path Algebras},
Lecture Notes in Math.\ \textbf{2191}, Springer, 2017.



\bibitem{Rufus}
R.\ Bowen and C.\ Series,
\textit{Markov maps associated with Fuchsian groups},
Publ.\ Math.\ Inst.\ Hautes \'Etudes Sci.\ \textbf{50} (1979), 153--170.

\bibitem{BJ}
O.\ Bratteli and P.\ E.\ T.\ Jorgensen,
\textit{Iterated function systems and permutation representations of the Cuntz algebra},
Mem.\ Amer.\ Math.\ Soc.\ \textbf{139} (1999), no.\ 663, 1--89.

\bibitem{danielcristobal}
C.\ G.\ Canto and D.\ Gon\c{c}alves,
\textit{Representations of relative Cohn path algebras},
J.\ Pure Appl.\ Algebra \textbf{224} (2020), no.\ 7, 106310.

\bibitem{CarlsenandLarsen}
T.\ M.\ Carlsen and N.\ S.\ Larsen,
\textit{Partial actions and KMS states on relative graph $C^*$-algebras},
J.\ Funct.\ Anal.\ \textbf{271} (2016), no.\ 8, 2090--2132.

\bibitem{GilleseDaniel}
G.\ G.\ de Castro and D.\ Gon\c{c}alves,
\textit{KMS and ground states on ultragraph $C^*$-algebras},
Integral Equations Operator Theory \textbf{90} (2018), no.\ 6, 63.

\bibitem{Danieleportugueses}
C.\ Correia Ramos, D.\ Gon\c{c}alves, N.\ Martins and P.\ R.\ Pinto,
\textit{Unifying interval maps and branching systems with applications to relative graph $C^*$-algebras},
J.\ Math.\ Anal.\ Appl.\ \textbf{519} (2023), 126757.

\bibitem{CMP}
C.\ Correia Ramos, N.\ Martins, P.\ R.\ Pinto and J.\ Sousa Ramos,
\textit{Cuntz-Krieger algebras representations from orbits of interval maps},
J.\ Math.\ Anal.\ Appl.\ \textbf{341} (2008), 825--833.

\bibitem{RMP5}
C.\ Correia Ramos, N.\ Martins and P.\ R.\ Pinto,
\textit{Interval maps from Cuntz--Krieger algebras},
J.\ Math.\ Anal.\ Appl.\ \textbf{374} (2011), 347--354.

\bibitem{RMP10}
C.\ Correia Ramos, N.\ Martins and P.\ R.\ Pinto,
\textit{Toeplitz algebras arising from escape points of interval maps},
Banach J.\ Math.\ Anal.\ \textbf{11} (2017), 536--553.

\bibitem{RMP11}
C.\ Correia Ramos, N.\ Martins and P.\ R.\ Pinto,
\textit{Escape dynamics for interval maps},
Discrete Contin.\ Dyn.\ Syst.\ \textbf{39} (2019), no.\ 11, 6241--6260.

\bibitem{CK}
J.\ Cuntz and W.\ Krieger,
\textit{A class of $C^*$-algebras and topological Markov chains},
Invent.\ Math.\ \textbf{56} (1980), 251--268.

\bibitem{DuPalle}
D.\ E.\ Dutkay and P.\ E.\ T.\ Jorgensen,
\textit{Wavelet constructions in non-linear dynamics},
Electron.\ Res.\ Announc.\ Amer.\ Math.\ Soc.\ \textbf{11} (2005), 21--33.


\bibitem{FGKP2016}
C.\ Farsi, E.\ Gillaspy, S.\ Kang and J.\ A.\ Packer,
\textit{Separable representations, KMS states, and wavelets for higher-rank graphs},
J.\ Math.\ Anal.\ Appl.\ \textbf{434} (2016), 241--270.

\bibitem{FGJKP2020}
C.\ Farsi, E.\ Gillaspy, P.\ Jorgensen, S.\ Kang and J.\ A.\ Packer,
\textit{Monic representations of finite higher-rank graphs},
Ergodic Theory Dynam.\ Systems \textbf{40} (2020), 1224--1252.





\bibitem{Daniel}
D.\ Gon\c{c}alves and D.\ Royer,
\textit{Unitary equivalence of representations of graph algebras and branching systems},
Funct.\ Anal.\ Appl.\ \textbf{45} (2011), 117--127.

\bibitem{GLR2016CMB}
D.\ Gon\c{c}alves, H.\ Li and D.\ Royer,
\textit{Faithful representations of graph algebras via branching systems},
Canad.\ Math.\ Bull.\ \textbf{59} (2016), 95--103.

\bibitem{DDinicio}
D.\ Gon\c{c}alves and D.\ Royer,
\textit{Graph $C^*$-algebras, branching systems and the Perron--Frobenius operator},
J.\ Math.\ Anal.\ Appl.\ \textbf{391} (2012), 457--465.

\bibitem{DDseparated}
D.\ Gon\c{c}alves and D.\ Royer,
\textit{Branching systems and representations of Cohn--Leavitt path algebras of separated graphs},
J.\ Algebra \textbf{422} (2015), 413--426.

\bibitem{DDH}
D.\ Gon\c{c}alves, H.\ Li and D.\ Royer,
\textit{Branching systems and general Cuntz--Krieger uniqueness theorem for ultragraph $C^*$-algebras},
Internat.\ J.\ Math.\ \textbf{27} (2016), no.\ 10, 1650083.

\bibitem{GLR2018GMJ}
D.\ Gon\c{c}alves, H.\ Li and D.\ Royer,
\textit{Branching systems for higher-rank graph $C^*$-algebras},
Glasg.\ Math.\ J.\ \textbf{60} (2018), 731--751.

\bibitem{OutroMarkov}
T.\ Jordan, S.\ Munday and T.\ Sahlsten,
\textit{Stability and perturbations of countable Markov maps},
Nonlinearity \textbf{31} (2018), 1351--1375.


\bibitem{Katsura}
T.\ Katsura, P.\ S.\ Muhly, A.\ Sims and M.\ Tomforde,
\textit{Ultragraph $C^*$-algebras via topological quivers},
Studia Math.\ \textbf{187} (2008), 137--155.

\bibitem{marcolli}
M.\ Marcolli and A.\ M.\ Paolucci,
\textit{Cuntz--Krieger algebras and wavelets on fractals},
Complex Anal.\ Oper.\ Theory \textbf{5} (2011), 41--81.


\bibitem{Tom}
P.\ S.\ Muhly and M.\ Tomforde,
\textit{Adding tails to $C^*$-correspondences},
Doc.\ Math.\ \textbf{9} (2004), 79--106.

\bibitem{Tom3}
M.\ Tomforde,
\textit{Simplicity of ultragraph algebras},
Indiana Univ.\ Math.\ J.\ \textbf{52} (2003), 901--925.



\bibitem{Tom2}
M.\ Tomforde,
\textit{A unified approach to Exel--Laca algebras and $C^*$-algebras associated to graphs},
J.\ Operator Theory \textbf{50} (2003), 345--368.

\bibitem{Sims}
A.\ Sims,
\textit{Relative Cuntz--Krieger algebras of finitely aligned higher-rank graphs},
Indiana Univ.\ Math.\ J.\ \textbf{55} (2006), 849--868.



\end{thebibliography}
\end{document}